\documentclass[noinfoline, 12pt]{imsart} 

\usepackage[utf8]{inputenc} 

\usepackage{geometry} 
\usepackage{amsthm, amsfonts, amssymb, amsmath,enumitem, natbib, bbm, mathabx, todonotes}
\usepackage{mathtools}
\mathtoolsset{showonlyrefs}
\newtheorem{defi}{Definition}[section]
\newtheorem{thm}{Theorem}[section]
\newtheorem{lem}{Lemma}[section]
\newtheorem{cor}{Corollary}[section]
\newtheorem{prop}{Proposition}[section]
\newtheorem{rem}{Remark}[section]
\providecommand{\norm}[1]{\left\lVert#1\right\rVert}
\DeclareMathOperator*{\argmin}{arg\,min}

\DeclareMathOperator*{\med}{\texttt{med}}
\DeclareRobustCommand{\rchi}{{\mathpalette\irchi\relax}}
\newcommand{\irchi}[2]{\raisebox{\depth}{$#1\chi$}}
\usepackage{graphicx} 
\usepackage{color, float}

\usepackage{booktabs} 
\usepackage{array} 
\usepackage{verbatim} 
\usepackage{subfig} 
\makeatletter
\def\namedlabel#1#2{\begingroup
    #2%
    \def\@currentlabel{#2}%
    \phantomsection\label{#1}\endgroup
}
\makeatother
\newcommand{\vertiii}[1]{{\left\vert\kern-0.25ex\left\vert\kern-0.25ex\left\vert #1
    \right\vert\kern-0.25ex\right\vert\kern-0.25ex\right\vert}}

\usepackage[nottoc,notlof,notlot]{tocbibind} 
\usepackage[titles,subfigure]{tocloft} 

\newcommand{\red}{\color{red}}
\newcommand{\blue}{\color{blue}}

\begin{document}
 \begin{frontmatter}
\title{All of Linear Regression}
\runtitle{All of Linear Regression}
\begin{aug}
  \author{\fnms{Arun K.}  \snm{Kuchibhotla}\ead[label=e1]{arunku@upenn.edu, buja.at.wharton@gmail.com}},
  \author{\fnms{Lawrence D.} \snm{Brown}},
  \author{\fnms{Andreas} \snm{Buja}},
  \and
  \author{\fnms{Junhui} \snm{Cai}}

  \runauthor{Kuchibhotla et al.}

  \affiliation{University of Pennsylvania}

  \address{University of Pennsylvania\\ \printead{e1}}

\end{aug}
\maketitle

\begin{abstract}
Least squares linear regression is one of the oldest and widely used data analysis tools. Although the theoretical analysis of ordinary least squares (OLS) estimator is as old, several fundamental questions are yet to be answered. Suppose regression observations $(X_1,Y_1),\ldots,(X_n,Y_n)\in\mathbb{R}^d\times\mathbb{R}$ (not necessarily independent) are available. Some of the questions we deal with are as follows: under what conditions, does the OLS estimator converge and what is the limit? What happens if the dimension is allowed to grow with $n$? What happens if the observations are dependent with dependence possibly strengthening with $n$? How to do statistical inference under these kinds of misspecification? What happens to OLS estimator under variable selection? How to do inference under misspecification and variable selection?

We answer all the questions raised above with one simple deterministic inequality which holds for any set of observations and any sample size. This implies that all our results are finite sample (non-asymptotic) in nature. At the end, one only needs to bound certain random quantities under specific settings of interest to get concrete rates and we derive these bounds for the case of independent observations. In particular the problem of inference after variable selection is studied, for the first time, when $d$, the number of covariates increases (almost exponentially) with sample size $n$. We provide comments on the ``right'' statistic to consider for inference under variable selection and efficient computation of quantiles.   
\end{abstract}

\end{frontmatter}

\section{Introduction}
Linear regression is one of the oldest and most widely practiced data analysis method. In many real data settings least squares linear regression leads to performance in par with state-of-the-art (and often far more complicated) methods while remaining amenable to interpretation. These advantages coupled with the argument ``all models are wrong'' warrants a detailed study of least squares linear regression estimator in settings that are close to the practical/realistic ones. Instead of proposing assumptions that we think are practical/realistic, we start with a clean slate. We start by not assuming anything about the observations $(X_1^{\top},Y_1)^{\top}, \ldots, (X_n^{\top},Y_n)^{\top}\in\mathbb{R}^d\times\mathbb{R}$ and study the OLS estimator $\hat{\beta}$ given by
\[
\hat{\beta} := \argmin_{\theta\in\mathbb{R}^d}\,\frac{1}{n}\sum_{i=1}^n (Y_i - X_i^{\top}\theta)^2,
\]
where $\argmin$ represents a $\theta$ at which the minimum is attained and this $\hat{\beta}$ may not be unique, in which case any of the minimizers is set as $\hat{\beta}$. This clean slate study should be compared to the usual assumption-laden approach where one usually starts by assuming that there exists a vector $\beta_0\in\mathbb{R}^d$ such that $Y_i = X_i^{\top}\beta_0 + \varepsilon_i$ for independent and identically distributed Gaussian homoscedastic errors $\varepsilon_1,\ldots,\varepsilon_n$. The classical linear regression setting (Gauss-Markov model) sometimes also assumes $X_1,\ldots,X_n$ are deterministic/non-stochastic. In this model, it is well-known that $\hat{\beta}$ has a normal distribution and is the best linear unbiased estimator (BLUE) for every sample size $n\ge d$.
\paragraph{Why is a clean slate study possible?} At first glance it might seem strange how a study without assumptions is possible. For a simple explanation, set
\begin{equation}\label{eq:Def-Gamma-Sigma}
\hat{\Gamma} := \frac{1}{n}\sum_{i=1}^n X_iY_i\quad\mbox{and}\quad \hat{\Sigma} := \frac{1}{n}\sum_{i=1}^n X_iX_i^{\top}.
\end{equation}
Now the vector $\hat{\beta}$ can be written as
\begin{equation}\label{eq:BetaHatGammaSigmaDef}
\hat{\beta} := \argmin_{\theta\in\mathbb{R}^d}\,-2\theta^{\top}\hat{\Gamma} + \theta^{\top}\hat{\Sigma}\theta,
\end{equation}
which implies that $\hat{\beta}$ is a minimizer of a (positive semi-definite) quadratic problem. Intuition suggests that if $\hat{\Gamma}\approx\Gamma$ and $\hat{\Sigma}\approx\Sigma$ then $\hat{\beta}$ is close to $\beta$ given by
\begin{equation}\label{eq:BetaGammaSigmaDef}
\beta := \argmin_{\theta\in\mathbb{R}^d}\,-2\theta^{\top}\Gamma + \theta^{\top}\Sigma\theta.
\end{equation}
A follow-up of this intuition suggests an explicit bound on $\|\hat{\beta} - \beta\|$ given bounds on $\|\hat{\Gamma} - \Gamma\|$ and $\|\hat{\Sigma} - \Sigma\|$, for (possibly different) norms $\|\cdot\|$. This viewpoint is usually seen in perturbation analysis of optimization problems; see~\cite{bonnans2013perturbation}. Note that~\eqref{eq:BetaHatGammaSigmaDef} can be seen as a perturbation of~\eqref{eq:BetaGammaSigmaDef}. Implementation of this program leads to our deterministic inequality and all subsequent results follow from this result as relatively simple corollaries.
\paragraph{Organization of the paper.} The remaining paper is organized as follows. We start, in Section~\ref{sec:Main-Deterministic-Ineq}, with a simple deterministic inequality that provides ``consistency'' and ``asymptotic normality'' of the OLS estimator $\hat{\beta}$. This will be a part survey with full proofs since similar results appeared before. We will describe explicit corollaries of this inequality for a Berry--Esseen type result for $\hat{\beta}$ that bounds the closeness of the distribution of $\hat{\beta}$ to that of a normal distribution; this is a finite sample result. In a way, this completes the study of OLS estimator in the clean slate setting because normal approximation is the crucial ingredient in statistical inference leading to confidence intervals and hypothesis tests; this discussion is given in Section~\ref{sec:Statistical-Inference}. The test statistics and confidence regions presented in this section are different from the ones used in the classical study. We chose to present the unconventional ones since they will be useful in the study of OLS estimator in presence of variable (or covariate) selection.

We then proceed to study OLS in presence of variable selection in Section~\ref{sec:OLS-Variable-Selection}. The setting here is that the analyst choses a subset of covariates (possibly depending on the data) and then consider the OLS estimator on that subset of covariates. Thanks to the deterministic inequality in Section~\ref{sec:Main-Deterministic-Ineq}, the results for this setting also follow directly. As a corollary, we also prove a Berry--Esseen type result uniformly over all subset of variables. We end Section~\ref{sec:OLS-Variable-Selection} with a discussion on how to perform statistical inference under variable selection in case observations are ``weakly'' dependent without stressing on details (about resampling). This discussion also includes the question of the ``right'' statistic to consider to inference under variable selection. All the results to this point will be deterministic, finite sample (or non-asymptotic). In Section~\ref{sec:Independence}, we provide explicit rate bounds for remainders in the deterministic inequalities from previous sections under independence of observations. This will complete the study of inference under variable selection, at least under independence, when the number of covariates is allowed to increase. We supplement these theoretical results with some numerical evidence in Section~\ref{sec:Simulations} where the proposed statistics for inference under variable selection are compared to the ones in the literature. The paper ends with a discussion and some comments on computation for inference under variable selection in Section~\ref{sec:Summary}.
\paragraph{Notation.} The following notation will be useful. For any vector $v\in\mathbb{R}^d$, $v^{\top}$ represents its transpose and $v_M\in\mathbb{R}^{|M|}$ for $M\subseteq\{1,2,\ldots,d\}$ represents the sub-vector of $v$ with entries in $M$. For instance $v = (4,3,2,1)^{\top}$ and $M = \{2,3\}$ then $v_M = (3,2)^{\top}$. Similarly for a symmetric matrix $A\in\mathbb{R}^{d\times d}$, $A_{M}\in\mathbb{R}^{|M|\times|M|}$ represents the sub-matrix of $A$ with entries in $M\times M$. The Euclidean norm in any dimension is given by $\|\cdot\|$. For any matrix $A$, let $\|A\|_{op}$ represents the operator norm of $A$, that is, $\|A\|_{op} = \sup_{\|\theta\| = 1}\|A\theta\|$. For any vector $\mu\in\mathbb{R}^q$ and any covariance matrix $\Omega\in\mathbb{R}^{q\times q}$, $N(\mu,\Omega)$ represents the (multivariate) normal distribution with mean $\mu$, covariance $\Omega$ and with some abuse of notation we also use $N(\mu,\Omega)$ to denote a random vector with that Gaussian distribution. For any covariance matrix $A$, $A^{1/2}$ represents the matrix square root and when we write $A^{-1}$ it is implicitly assumed that $A$ is invertible with inverse $A^{-1}.$ The identity matrix in dimension $q$ is given by $I_q$. Further for any covariance matrix $A\in\mathbb{R}^{q\times q}$ and vector $x\in\mathbb{R}^{q}$, $\|x\|_A := \sqrt{x^{\top}Ax}$.
\section{Main Deterministic Inequality}\label{sec:Main-Deterministic-Ineq}
Recall the quantities $\hat{\Gamma}$ and~$\hat{\Sigma}$ defined in~\eqref{eq:Def-Gamma-Sigma}. The following result proves deterministic bounds on estimation error and linear representation error for the OLS estimator $\hat{\beta}$. Let $(t)_+ := \max\{0, t\}$ for $t\in\mathbb{R}$ and for any $\Sigma\in\mathbb{R}^{d\times d}$, set
\begin{equation}\label{eq:D-Sigma-Definition}\textstyle
\mathcal{D}^{\Sigma} := \|\Sigma^{-1/2}\hat{\Sigma}\Sigma^{-1/2} - I_d\|_{op}.
\end{equation}
\begin{thm}[Deterministic Inequality]\label{thm:Deterministic-Ineq}
For any symmetric matrix $\Sigma\in\mathbb{R}^{d\times d}$and for any vector $\beta\in\mathbb{R}^d$, we have
\begin{equation}\label{eq:Consistency}
(1 + \mathcal{D}^{\Sigma})^{-1}{\|\Sigma^{-1}(\hat{\Gamma} - \hat{\Sigma}\beta)\|_{\Sigma}} ~\le~ \|\hat{\beta} - \beta\|_{\Sigma} ~\le~ (1 - \mathcal{D}^{\Sigma})_+^{-1}{\|\Sigma^{-1}(\hat{\Gamma} - \hat{\Sigma}\beta)\|_{\Sigma}}.
\end{equation}
Furthermore,
\begin{equation}\label{eq:Linear-Representation}\textstyle
\|\hat{\beta} - \beta - \Sigma^{-1}(\hat{\Gamma} - \hat{\Sigma}\beta)\|_{\Sigma} ~\le~ {\mathcal{D}^{\Sigma}}(1 - \mathcal{D}^{\Sigma})_+^{-1}\|\Sigma^{-1}(\hat{\Gamma} - \hat{\Sigma}\beta)\|_{\Sigma}. 
\end{equation} 
\end{thm} 
\begin{proof}
From the definition of $\hat{\beta}$, we have the normal equations $\hat{\Sigma}\hat{\beta} = \hat{\Gamma}.$
Subtracting $\hat{\Sigma}\beta\in\mathbb{R}^d$ from both sides, we get
$\hat{\Sigma}(\hat{\beta} - \beta) = \hat{\Gamma} - \hat{\Sigma}\beta,$
which is equivalent to
\[\textstyle
(\Sigma^{-1/2}\hat{\Sigma}\Sigma^{-1/2})\Sigma^{1/2}(\hat{\beta} - \beta) ~=~ \Sigma^{-1/2}(\hat{\Gamma} - \hat{\Sigma}\beta).
\]
Adding and subtracting $I_d$ from the parenthesized term with further rearrangement, we get
$\Sigma^{1/2}(\hat{\beta} - \beta) - \Sigma^{-1/2}(\hat{\Gamma} - \hat{\Sigma}\beta) = \left(I_d - \Sigma^{-1/2}\Sigma\Sigma^{-1/2}\right)\Sigma^{1/2}(\hat{\beta} - \beta).$
Taking Euclidean norm on both sides yields
\begin{align}\label{eq:Almost-Final-Bound}
\begin{split}
\left\|\Sigma^{1/2}\big[\hat{\beta} - \beta - \Sigma^{-1}(\hat{\Gamma} - \hat{\Sigma}\beta)\big]\right\| ~&=~ \|(I_d - \Sigma^{-1/2}\hat{\Sigma}\Sigma^{-1/2})\Sigma^{1/2}(\hat{\beta} - \beta)\|\\
~&\le~ \|I_d - \Sigma^{-1/2}\hat{\Sigma}\Sigma^{-1/2}\|_{op}\|\Sigma^{1/2}(\hat{\beta} - \beta)\|\\
~&=~ \mathcal{D}^{\Sigma}\|\hat{\beta} - \beta\|_{\Sigma},
\end{split}
\end{align}
where the inequality follows from the definition of the operator norm, $\|\cdot\|_{op}$. 
Triangle inequality shows
$|\|\hat{\beta} - \beta\|_{\Sigma} - \|\Sigma^{-1}(\hat{\Gamma} - \hat{\Sigma}\beta)\|_{\Sigma}| ~\le~ \|\hat{\beta} - \beta - \Sigma^{-1}(\hat{\Gamma} - \hat{\Sigma}\beta)\|_{\Sigma},$
which when combined with~\eqref{eq:Almost-Final-Bound} yields
\[
\|\hat{\beta} - \beta\|_{\Sigma} ~\le~ \frac{\|\Sigma^{-1}(\hat{\Gamma} - \hat{\Sigma}\beta)\|_{\Sigma}}{(1 - \mathcal{D}^{\Sigma})_+}\quad\mbox{and}\quad \|\hat{\beta} - \beta\|_{\Sigma} ~\ge~ \frac{\|\Sigma^{-1}(\hat{\Gamma} - \hat{\Sigma}\beta)\|_{\Sigma}}{1 + \mathcal{D}^{\Sigma}}. 
\]
These inequalities prove~\eqref{eq:Consistency} and when combined with~\eqref{eq:Almost-Final-Bound} implies~\eqref{eq:Linear-Representation}.
\end{proof}
Theorem~\ref{thm:Deterministic-Ineq} is a very general result that holds for any set of observations (not even necessarily random). It is noteworthy that the result holds for any symmetric matrix $\Sigma$ and ``target'' vector $\beta\in\mathbb{R}^d$. A canonical choice of $\Sigma$ and $\beta$ are given by
\begin{equation}\label{eq:Canonical-choice}
\Sigma := \mathbb{E}[\hat{\Sigma}] = \frac{1}{n}\sum_{i=1}^n \mathbb{E}\left[X_iX_i^{\top}\right],\quad\mbox{and}\quad \beta := \argmin_{\theta\in\mathbb{R}^d}\,\frac{1}{n}\sum_{i=1}^n \mathbb{E}\left[(Y_i - X_i^{\top}\theta)^2\right].
\end{equation}
It is important to note here that just by taking expectations we do not necessarily require all observations to be (non-trivially) random; even fixed numbers are random with a degenerate distribution. For example, in the classical linear model $X_i$'s are treated fixed and non-stochastic in which case $\Sigma = \hat{\Sigma}$ and hence $\mathcal{D}^{\Sigma} = 0$. Moreover, we neither require any specific dependence structure on the observations nor any specific scaling of dimension $d$ with $n$. By a careful inspection of the proof and a slight adjustment of $\mathcal{D}^{\Sigma}$ in~\eqref{eq:D-Sigma-Definition}, it is possible to prove the result for $\|\hat{\beta} - \beta\|$, the usual Euclidean norm, instead of $\|\hat{\beta} - \beta\|_{\Sigma}$. The added advantage of using $\|\cdot\|_{\Sigma}$ is affine invariance of the result. 
\paragraph{Flexibility in the Choice of $\Sigma$ and $\beta$.}For most purposes the canonical choices of $\Sigma, \beta$ in~\eqref{eq:Canonical-choice} suffice but for some applications involving sub-sampling and cross-validation, the flexibility in choosing $\Sigma, \beta$ helps. For instance, consider the OLS estimator constructed based on the first $n-1$ observations, that is,
\[
\hat{\beta}_{-n} := \argmin_{\theta\in\mathbb{R}^d}\,{(n-1)^{-1}}\sum_{i=1}^{n-1} (Y_i - X_i^{\top}\theta)^2 = \argmin_{\theta\in\mathbb{R}^d}\, -2\theta^{\top}\hat{\Gamma}_{-n} + \theta^{\top}\hat{\Sigma}_{-n}\theta,
\]
where
$\hat{\Gamma}_{-n} := (n-1)^{-1}\sum_{i=1}^{n-1} X_iY_i$ and $\hat{\Sigma}_{-n} := (n-1)^{-1}\sum_{i=1}^{n-1} X_iX_i^{\top}.$
It is of natural interest to compare $\hat{\beta}_{-n}$ with $\hat{\beta}$ rather than the canonical choice of $\beta$. In this case $\Sigma$ is taken to be $\hat{\Sigma}$ which is much closer to $\hat{\Sigma}_{-n}$ than $\mathbb{E}[\hat{\Sigma}_{-n}]$:
\[
\hat{\Sigma}_{-n} = \frac{n}{n-1}\left(\hat{\Sigma} - n^{-1}X_nX_n^{\top}\right)\;\;\Rightarrow\;\; \hat{\Sigma}^{-1/2}\hat{\Sigma}_{-n}\hat{\Sigma}^{-1/2} = \frac{nI_d}{n-1} - \frac{\hat{\Sigma}^{-1/2}X_nX_n^{\top}\hat{\Sigma}^{-1/2}}{n-1}.
\]
Hence $\|\hat{\Sigma}^{-1/2}\hat{\Sigma}_{-n}\hat{\Sigma}^{-1/2} - I_d\|_{op} \le (n-1)^{-1}[1 + \|\hat{\Sigma}^{-1/2}X_n\|^2].$

\subsection[Consistency]{Consistency of $\hat{\beta}$} If $\mathcal{D}^{\Sigma} < 1$, then inequalities in~\eqref{eq:Consistency} provides both upper bounds and lower bounds on the estimation error $\|\hat{\beta} - \beta\|_{\Sigma}$ that match up to a constant multiple. This allows one to state that necessary and sufficient condition for convergence of $\|\hat{\beta} - \beta\|_{\Sigma}$ to zero is $\|\Sigma^{-1}(\hat{\Gamma} - \hat{\Sigma}\beta)\|_{\Sigma}$ has to converge to zero. Note that with the choices in~\eqref{eq:Canonical-choice} $\Sigma^{-1}(\hat{\Gamma} - \hat{\Sigma}\beta)$ is a mean zero random vector obtained by averaging $n$ random vectors and hence ``weak'' dependence implies convergence of covariance to zero implying convergence to zero. This implies consistency of the OLS estimator $\hat{\beta}$ to $\beta$:
\begin{cor}[Consistency]\label{cor:Consistency}
If $\mathcal{D}^{\Sigma} < 1$ and $\|\Sigma^{-1}(\hat{\Gamma} - \hat{\Sigma}\beta)\|_{\Sigma}$ converges to zero in probability then $\|\hat{\beta} - \beta\|_{\Sigma}$ converges to zero in probability.
\end{cor}

Turning to inequality~\eqref{eq:Linear-Representation}, note that if $\mathcal{D}^{\Sigma} \to 0$ (in appropriate sense) then inequality~\eqref{eq:Linear-Representation} provides an expansion of $\hat{\beta} - \beta$ since the remainder (the right hand side of~\eqref{eq:Linear-Representation}) is of smaller order than $\hat{\beta} - \beta$. Observe that
$\Sigma^{-1}(\hat{\Gamma} - \hat{\Sigma}\beta) = n^{-1}\sum_{i=1}^n \Sigma^{-1}X_i(Y_i - X_i^{\top}\beta),$
and hence~\eqref{eq:Linear-Representation} shows that $\hat{\beta} - \beta$ behaves like an average (a linear functional) up to a lower order term. The claim
\begin{equation}\label{eq:Influence-function-expansion}
\sqrt{n}(\hat{\beta} - \beta) = \frac{1}{\sqrt{n}}\sum_{i=1}^n \Sigma^{-1}X_i(Y_i - X_i^{\top}\beta) + o_p(1),
\end{equation}
is usually referred to as an influence function expansion or a linear approximation result. This plays a pivotal role in statistical inference because of the following reason. Ignoring the $o_p(1)$ term, the right hand side of~\eqref{eq:Influence-function-expansion} is a mean zero (scaled) average of random vectors which, under almost all dependence settings of interest, converges to a normal distribution if the dimension $d$ is fixed or even diverging ``slow enough''. This implies that $\sqrt{n}(\hat{\beta} - \beta)$ has an asymptotic normal distribution and an accessible estimator of the (asymptotic) variance implies confidence intervals/regions and hypothesis tests. This discussion is in asymptotic terms and can be made explicitly finite sample which we do in the following subsection with inference related details in the next section. 

\subsection{Normal Approximation: Berry--Esseen Result} In the following corollary (of Theorem~\ref{thm:Deterministic-Ineq}), we prove a bound on closeness of distribution of $\hat{\beta} - \beta$ to a normal distribution. We need some definitions. Set
\begin{equation}\label{eq:Berry-Esseen-Average}\textstyle
\Delta_n := \sup_{A\in\mathcal{C}_d}\left|\mathbb{P}(\Sigma^{-1/2}(\hat{\Gamma} - \hat{\Sigma}\beta) \in A) - \mathbb{P}\left(N(0, K) \in A\right)\right|,
\end{equation}
where $\mathcal{C}_d$ represents the set of all convex sets in $\mathbb{R}^d$ and $K := \mbox{Var}(\Sigma^{-1/2}(\hat{\Gamma} - \hat{\Sigma}\beta))$.
For any matrix $A$, let $\|A\|_{HS}$ represent the Hilbert-Schmidt (or Frobenius) norm, that is, $\|A\|_{HS}^2 := \sum_{i,j} A^2(i,j)$. Also, for any positive semi-definite matrix $A$, let $\|A\|_{*}$ denote the nuclear norm of the matrix $A$. 
\begin{cor}[Berry--Esseen bound for OLS]\label{cor:BEOLS}
Fix any $\eta\in(0, 1)$. Then there exists universal constants $c_1, c_2 > 0$ such that for all $n\ge 1$,
\begin{align*}
&\sup_{A\in\mathcal{C}_d}\left|\mathbb{P}(\hat{\beta} - \beta\in A) - \mathbb{P}\left(N(0, \Sigma^{-1/2}K\Sigma^{-1/2}) \in A\right)\right|\\ &\qquad\le 4\Delta_n + 2n^{-1} + c_2\|K_n^{-1}\|_*^{1/4}r_n\eta + \mathbb{P}\left(\mathcal{D}^{\Sigma} > \eta\right),
\end{align*}
where recall $\mathcal{D}^{\Sigma}$ from~\eqref{eq:D-Sigma-Definition} and $r_n := c_1^{-1}\|K^{1/2}\|_{op}\sqrt{\log n} + \|K^{1/2}\|_{HS}$.
\end{cor} 
The proof of the corollary can be found in Appendix~\ref{AppSec:ProofBEOLS} and it does not require~\eqref{eq:Canonical-choice}. The proof of normal approximation for multivariate minimum contrast estimators in~\cite{Pfanzagl72} is very similar to that of Corollary~\ref{cor:BEOLS}. Like Theorem~\ref{thm:Deterministic-Ineq}, Corollary~\ref{cor:BEOLS} is also a finite sample result that does not assume any specific dependence structure on the observations. The quantity $\Delta_n$ in~\eqref{eq:Berry-Esseen-Average} is a quantification of convergence of right hand side of~\eqref{eq:Influence-function-expansion} to a normal distribution and is bounded by the available multivariate Berry--Esseen bounds. Such bounds for independent (but not necessarily identically distributed) random vectors can be found in~\cite{Bent04} and~\cite{raivc2018multivariate}. For dependent settings, multivariate Berry--Esseen bounds are hard to find but univariate versions available (in~\cite{RomanoWolf00} and~\cite{Siegfried09}) can be extended to multivariate versions by the characteristic function method and smoothing inequalities. In this respect, we note here that the proof of Corollary~\ref{cor:BEOLS} can be extended to prove a normal approximation result for $\alpha^{\top}(\hat{\beta} - \beta)$ for any specific direction $\alpha\in\mathbb{R}^d$ and for this univariate random variable results from above references apply directly. Finally to get concrete rates from the bound in Corollary~\ref{cor:BEOLS}, we only need to choose $\eta\in(0, 1)$ and for this we need to control the tail probability of $\mathcal{D}^{\Sigma}$ in~\eqref{eq:D-Sigma-Definition}. There are two choices for this. Firstly, assuming moment bounds for $X_1,\ldots,X_n$, it is possible to get a tail bound for $\mathcal{D}^{\Sigma}$ under reasonable dependence structures; see~\cite{Uniform:Kuch18} and \citet[for independence case]{koltchinskii2017a}. Secondly, one can use a Berry--Esseen type result~\citep{koltchinskii2017b} for $\mathcal{D}^{\Sigma}$ which also implies an exponential tail bound up to an analogue of $\Delta_n$ term.
\paragraph{Glimpse of the Rates.} Assuming observations $(X_i,Y_i),1\le i\le n$ are sufficiently weakly dependent and have enough moments, it can be proved that
\begin{equation}\label{eq:Rates-Glimpse}
\|\Sigma^{-1}(\hat{\Gamma} - \hat{\Sigma}\beta)\|_{\Sigma} + \|\Sigma^{-1/2}\hat{\Sigma}\Sigma^{-1/2} - I_d\|_{op} = O_p(1)\sqrt{\frac{p}{n}}.
\end{equation}
See Section~\ref{sec:Independence}. For concrete rates in normal approximation, observe that
\[
\|K^{-1}\|_{*}^{1/4} \le p^{1/4}\|K^{-1}\|_{op}^{1/4}\quad\mbox{and}\quad \|K^{1/2}\|_{HS} = \sqrt{\mbox{tr}(K)} \le p^{1/2}\|K\|_{op}^{1/2}.
\]
This implies that 
\[
\|K^{-1}\|_*^{1/4}r_n ~=~ p^{1/4}\|K^{-1}\|^{1/4}_{op}~\times~ O(\|K\|_{op}^{1/2}\sqrt{\log n} + p^{1/2}\|K\|_{op}^{1/2}).
\] 
Under weak enough dependence structure, $\Sigma^{1/2}(\hat{\Gamma} - \hat{\Sigma}\beta)$ is $O_p(n^{-1/2})$ in any fixed direction and hence $\|K\|_{op} = O_p(n^{-1})$ where, recall, $K = \mbox{Var}(\Sigma^{-1/2}(\hat{\Gamma} - \hat{\Sigma}\beta))$. Assuming $\|K^{-1}\|_{op} \asymp \|K\|_{op}^{-1}$, we get $\|K^{-1}\|_*^{1/4}r_n = O(n^{-1/4}[p^{1/4}\sqrt{\log n} + p^{3/4}])$. In the best case scenario $\Delta_n \ge O(p^{7/4}n^{-1/2})$ and hence to match this rate, we can to take $\eta = O(n^{-1/4})$ which is a permissible choice under~\eqref{eq:Rates-Glimpse}. Hence we can claim
\begin{equation}\label{eq:Vague-Rates-Normal-Approximation}
\sup_{A\in\mathcal{C}_d}\left|\mathbb{P}(\hat{\beta} - \beta\in A) - \mathbb{P}\left(N(0, \Sigma^{-1/2}K\Sigma^{-1/2}) \in A\right)\right| = O(1)\frac{p^{7/4}}{n^{1/2}}.
\end{equation}
We have intentionally left the conditions vague which will be cleared in Section~\ref{sec:Independence}.
\paragraph{The Curious Case of Fixed Covariates.} In the conventional linear models theory, the covariates are treated fixed/non-stochastic. Since our results are deterministic in nature, this distinction does not matter for the validity of our results. However, in case of fixed covariates the canonical choices for $\Sigma, \beta$ mentioned above result in simpler results. For instance, it is clear that non-stochastic covariates leads to $\Sigma = \hat{\Sigma}$, both of which are non-stochastic, and hence
$\mathcal{D}^{\Sigma} = 0.$
Theorem~\ref{thm:Deterministic-Ineq} now implies that
$\|\hat{\beta} - \beta - \hat{\Sigma}^{-1}(\hat{\Gamma} - \hat{\Sigma}\beta)\|_{\hat{\Sigma}} = 0,$
or equivalently, $\hat{\beta} - \beta = \hat{\Sigma}^{-1}(\hat{\Gamma} - \hat{\Sigma}\beta)$ which is trivial from the definition of $\hat{\beta}$. Further from Corollary~\ref{cor:BEOLS}, we get
\[
\sup_{A\in\mathcal{C}_d}|\mathbb{P}(\hat{\beta} - \beta\in A) - \mathbb{P}(N(0, \Sigma^{-1/2}K\Sigma^{-1/2}) \in A)| \le 4\Delta_n + 2n^{-1},
\]
since $\eta$ can be taken to be zero in limit. In fact a careful modification of the proof leads to a sharper right hand side as $\Delta_n$. These calculations hint at a previously unnoticed phenomenon: The bounds for random covariates are inherently larger than those for fixed covariates (although they are all of same order). A similar statement also holds when some of the covariates are fixed but others are random (the bounds have extra terms only for random set of covariates). This phenomenon means that, when working with finite samples, the statistical conclusions can be significantly distorted depending on whether the covariates are treated fixed or random. Here it is worth mentioning that the canonical choice of $\beta$ changes depending on whether covariates are treated random or fixed. If the covariates are fixed, then the canonical choice $\beta$ is
$\beta = (n^{-1}\sum_{i=1}^n x_ix_i^{\top})^{-1}(n^{-1}\sum_{i=1}^n x_i\mathbb{E}[Y_i]),$ where we write $x_i$ (rather than $X_i$) to represent fixed nature of covariates. If the covariates are random, then the canonical choice $\beta$ is
$\beta = (n^{-1}\sum_{i=1}^n \mathbb{E}[X_iX_i^{\top}])^{-1}(n^{-1}\sum_{i=1}^n \mathbb{E}[X_iY_i]).$
\section{Statistical Inference for the OLS estimator}\label{sec:Statistical-Inference}
Given that the distribution of $\hat{\beta} - \beta$ is close to a mean zero Gaussian, inference follows if the variance of the Gaussian can be estimated. The variance of the Gaussian is given by
\begin{equation}\label{eq:V-definition}\textstyle
\Sigma^{-1}V\Sigma^{-1} := \Sigma^{-1}\mbox{Var}\left(n^{-1}\sum_{i=1}^n X_i(Y_i - X_i^{\top}\beta)\right)\Sigma^{-1},
\end{equation}
which is, sometimes, referred to as the sandwich variance. The two ends of the variance $\Sigma^{-1}$ can be estimated by $\hat{\Sigma}^{-1}$. The only troublesome part is the ``meat'' part which is the variance of a mean zero average. Estimation of this part requires an understanding of the dependence structure of observations. For instance if the observations are independent then we can readily write
\begin{equation}\label{eq:Conservative-Variance}
V = \frac{1}{n^2}\sum_{i=1}^n \mbox{Var}(X_i(Y_i - X_i^{\top}\beta)) ~\preceq~ \frac{1}{n^2}\sum_{i=1}^n \mathbb{E}\left[X_iX_i^{\top}(Y_i - X_i^{\top}\beta)^2\right]. 
\end{equation}
The inequality above is the matrix inequality representing the difference of matrices is positive semi-definite. A strict inequality above can hold since the observations need not satisfy $\mathbb{E}[X_i(Y_i - X_i^{\top}\beta)] = 0$. (The definition of $\beta$ only implies $\sum_{i=1}^n \mathbb{E}[X_i(Y_i - X_i^{\top}\beta)] = 0$.) The last term on the right of~\eqref{eq:Conservative-Variance} can be estimated by $n^{-2}\sum_{i=1}^n X_iX_i^{\top}(Y_i - X_i^{\top}\hat{\beta})^2$ (obtained by removing the expectation and then replacing $\beta$ by $\hat{\beta}$). This leads to asymptotically conservative inference for $\beta$ and it can be proved that asymptotically exact inference is impossible without further assumptions such as $\mathbb{E}[X_i(Y_i - X_i^{\top}\beta)] = 0$ for all $i$; see~\citet[Proposition 3.5]{Bac16} for an impossibility result. Instead if the observations are not independent but $m$-dependent, then the first equality of~\eqref{eq:Conservative-Variance} does not hold and a correction is needed involving the covariances of different summands; see~\cite{White2001} for details under specific dependence structures. 

Once an estimator (possibly conservative) $\hat{\Sigma}^{-1}\hat{V}\hat{\Sigma}^{-1}$ of the variance is available, a (possibly conservative) $(1 - \alpha)$-confidence region for $\beta\in\mathbb{R}^d$ can be obtained as
\begin{equation}\label{eq:Chi-Square-Region}
\hat{\mathcal{R}}_{2,\alpha} := \{\theta\in\mathbb{R}^d:\,(\hat{\beta} - \theta)^{\top}\hat{\Sigma}\hat{V}^{-1}\hat{\Sigma}(\hat{\beta} - \theta) \le \rchi^2_{d,\alpha}\},
\end{equation}
where $\rchi^2_{d,\alpha}$ represents the $(1 - \alpha)$-th quantile of the chi-square distribution with $d$ degrees of freedom. If $\hat{V}$ is an asymptotically conservative estimator for $V$ that is $\hat{V} \to \bar{V}$ (in an appropriate sense) and $\bar{V} \succeq V$, then
\begin{align*}
&\mathbb{P}(N(0, \Sigma^{-1}V\Sigma^{-1})^{\top}\hat{\Sigma}\hat{V}^{-1}\hat{\Sigma}N(0, \Sigma^{-1}V\Sigma^{-1}) \le \rchi^2_{d,\alpha})\\ 
&\qquad\to \mathbb{P}(N(0, \bar{V}^{-1/2}V\bar{V}^{-1/2})^{\top}N(0, \bar{V}^{-1/2}V\bar{V}^{-1/2})) \ge 1 - \alpha,
\end{align*}
where strict inequality holds if $\bar{V} \succ V$; the inequality above is true because of Anderson's lemma~\citep[Corollary 3]{Anderson55} and it may not be true for non-symmetric confidence regions.
An alternate $(1-\alpha)$-confidence region for $\beta$ is
\begin{equation}\label{eq:Max-t-Region}\textstyle
\hat{\mathcal{R}}_{\infty,\alpha} := \left\{\theta\in\mathbb{R}^d:\,\max_{1\le j\le d}\left|{\widehat{AV}_j}^{-1/2}(\hat{\beta}_j - \theta_j){}\right| \le z_{\infty,\alpha}\right\},
\end{equation}
where $\widehat{AV}_j$ represents the $j$-th diagonal entry of the variance estimator $\hat{\Sigma}^{-1}\hat{V}\hat{\Sigma}^{-1}$ and $z_{\infty,\alpha}$ is the $(1 - \alpha)$-th quantile of
$\max_{1\le j\le d}|AV_j^{-1/2}{N(0, AV)_j}|,$
with $AV\in\mathbb{R}^{d\times d}$ represents the variance matrix $\Sigma^{-1}V\Sigma^{-1}$. 

Hypothesis tests for $\beta\in\mathbb{R}^d$ can also be performed based on the statistics used in~\eqref{eq:Chi-Square-Region} and~\eqref{eq:Max-t-Region}. It is easy to verify that neither statistic uniformly beats the other in terms of power. The tests for a single coordinate $\beta_j$ are easy to obtain from the statistic $(\hat{\beta}_j - \beta_j)/\widehat{AV}_j^{1/2}$ which is close to a standard normal random variable. 

The advantage of $\hat{\mathcal{R}}_{\infty,\alpha}$ over $\hat{\mathcal{R}}_{2,\alpha}$ is that it leads to a rectangular region and hence easily interpretable inference for coordinates of $\beta$. The confidence region $\hat{\mathcal{R}}_{2,\alpha}$ which is elliptical makes this interpretation difficult.

Inference based on a closed form variance estimator can be thought of as a direct method and is, in general, hard to extend to general dependence structures. A safe choice and a more unified way of estimating the variance is by the use of some resampling scheme. Bootstrap and subsampling or their block versions are robust to slight changes in dependence structures and are more widely applicable. The literature along these lines is so vast to review and we refer the reader to~\cite{kunsch1989jackknife}, \cite{liu1992moving}, \cite{politis1994large}, \cite{lahiri1999theoretical} for general block sampling techniques for variance/distribution estimation. Finite sample study of direct method is easy while such a study for resampling methods (under dependence) is yet non-existent.
\section{OLS Estimator under Variable Selection}\label{sec:OLS-Variable-Selection}
Having understood the properties of the OLS estimator obtained from the full set of covariates, we now proceed to the practically important aspect of OLS under variable selection. More often than not is the case that the set of covariates in the final reported model is not the same as the full set of covariates and more concernedly the final set of covariates is chosen based on the data at hand. For concreteness, let $\hat{M}\subseteq\{1,2,\ldots,d\}$ represent the set of covariates selected and let $\hat{\beta}_{\hat{M}}$ represent the OLS estimator constructed based on covariate (indices) in $\hat{M}$. More generally for any set $M\subseteq\{1,2,\ldots,d\}$, let $\hat{\beta}_M$ represent the OLS estimator from covariates in $M$, that is,
\[
\hat{\beta}_{M} := \argmin_{\theta\in\mathbb{R}^{|M|}}\,\sum_{i=1}^n (Y_i - X_{i,M}^{\top}\theta)^2.
\]
The aim of this section is to understand the properties of $\hat{\beta}_{\hat{M}}$ (irrespective of how $\hat{M}$ is chosen). This problem further highlights the strength of the deterministic inequality in Theorem~\ref{thm:Deterministic-Ineq} which applies irrespective of randomness of $\hat{M}$. Define for any $M\subseteq\{1,2,\ldots,d\}$, the canonical ``target'' for OLS estimator $\hat{\beta}_M$ as
\[
\beta_M := \argmin_{\theta\in\mathbb{R}^p}\,\sum_{i=1}^n\mathbb{E}\left[(Y_i - X_{i,M}^{\top}\theta)^2\right].
\]
Also, define $\mathcal{D}_M^{\Sigma} := \|\Sigma_M^{-1/2}\hat{\Sigma}_M\Sigma_M^{-1/2} - I_{|M|}\|_{op}.$
where recall $\Sigma_M$ (and $\hat{\Sigma}_M$) represents the submatrix of $\Sigma$ (and $\hat{\Sigma}$). Recall $(t)_+ := \max\{0, t\}$.
\begin{cor}\label{cor:Uniform-in-Submodel}
For any $\hat{M}$, we have
\[
\big\|\hat{\beta}_{\hat{M}} - \beta_{\hat{M}} - \Sigma_{\hat{M}}^{-1}(\hat{\Gamma}_{\hat{M}} - \hat{\Sigma}_{\hat{M}}\beta_{\hat{M}})\big\|_{\Sigma_{\hat{M}}} ~\le~ \frac{\mathcal{D}^{\Sigma}_{\hat{M}}}{1 - \mathcal{D}^{\Sigma}_{\hat{M}}}\big\|\Sigma_{\hat{M}}^{-1}(\hat{\Gamma}_{\hat{M}} - \hat{\Sigma}_{\hat{M}}\beta_{\hat{M}})\big\|_{\Sigma_{\hat{M}}}.
\]
More generally, for all $M\subseteq\{1,2,\ldots,d\}$ (simultaneously), we have
\begin{equation}\label{eq:Inf-Expansion-sub-model}
\big\|\hat{\beta}_{M} - \beta_{M} - \Sigma_{M}^{-1}(\hat{\Gamma}_{M} - \hat{\Sigma}_{M}\beta_{M})\big\|_{\Sigma_{M}} ~\le~ \frac{\mathcal{D}_M^{\Sigma}}{(1 - \mathcal{D}_M^{\Sigma})_+}\big\|\Sigma_{M}^{-1}(\hat{\Gamma}_{M} - \hat{\Sigma}_{M}\beta_{M})\big\|_{\Sigma_{M}}.
\end{equation}
\end{cor}
Corollary~\ref{cor:Uniform-in-Submodel} follows immediately from Theorem~\ref{thm:Deterministic-Ineq} and for simplicity it is stated with $\Sigma, \beta$ choices in~\eqref{eq:Canonical-choice} but other choices are possible. The first inequality in the corollary proves an influence function type expansion for the estimator $\hat{\beta}_{\hat{M}}$ around (a possibly random) target vector $\beta_{\hat{M}}$. In order to prove convergence of the remainder in this expansion to zero, one needs to control $\mathcal{D}_{\hat{M}}$ which can be a bit complicated to deal with directly. With some information on how ``strongly'' dependent $\hat{M}$ is on the data, such a direct approach can be worked out; see~\citet[Proposition 1]{pmlr-v51-russo16},~\cite{jiao2017generalizations}. If no information other than the fact that $\hat{M}\in\mathcal{M}$ for some set, $\mathcal{M}$, of subsets of covariates, then we have
\begin{equation}\label{eq:Union-bound}
\mathcal{D}_{\hat{M}} ~\le~ U_{\hat{M}}~\times~\max_{M\in\mathcal{M}}\,\frac{\mathcal{D}_M}{U_M},
\end{equation} 
for any set of (non-stochastic) numbers $\{U_M:\,M\in\mathcal{M}\}$; $U_M$ usually converges to zero at rate $\sqrt{|M|\log(ed/|M|)/n}$; see Proposition~\ref{prop:Rates-D_M-Gamma_M}. Some examples of $\mathcal{M}$ include $\mathcal{M}_{\le k} := \{M\subseteq\{1,\ldots,d\}:1 \le |M| \le k\},\;\mathcal{M}_{= k} := \{M\subseteq\{1,\ldots,d\}:|M| = k\},$ for some $k\ge1$. Note that the maximum on the right hand side of~\eqref{eq:Union-bound} is random only through $\hat{\Sigma}$ (dissolving the randomness in $\hat{M}$ into the maximum over $\mathcal{M}$). We will take this indirect approach in our study since we do not want to make any assumption on how the model $\hat{M}$ is obtained which might as well be adversarial. Further note that~\eqref{eq:Union-bound} is tight (in that it cannot be improved) in an agnostic setting since one can take $\hat{M}$ such that $\mathcal{D}_{\hat{M}}/U_{\hat{M}} = \max_{M\in\mathcal{M}}\,\mathcal{D}_M/U_M$. We take the same indirect approach to bound $\|\Sigma_M^{-1}(\hat{\Gamma}_M - \hat{\Sigma}_M\beta_M)\|_{\Sigma_M}$ over $M\in\mathcal{M}$. These bounds prove consistency and linear representation error bounds for the OLS estimator under variable selection. Similar results can be derived for other modifications of OLS estimator such as transformations.
\subsection[Consistency under Variable Selection]{Consistency of $\hat{\beta}_{\hat{M}}$} From Corollary~\ref{cor:Uniform-in-Submodel} it is easy to prove the following corollary (similar to Corollary~\ref{cor:Consistency}) for consistency.
\begin{cor}[Consistency of $\hat{\beta}_{\hat{M}}$]\label{cor:Consistency-variable-selection}
If $\mathcal{D}_{\hat{M}}^{\Sigma} < 1$ and $\|\Sigma_{\hat{M}}^{-1}(\hat{\Gamma}_{\hat{M}} - \hat{\Sigma}_{\hat{M}}\beta_{\hat{M}})\| \to 0$ in probability, then $\|\hat{\beta}_{\hat{M}} - \beta_{\hat{M}}\|_{\Sigma_{\hat{M}}}$ converges to zero in probability.
\end{cor}
The conditions of Corollary~\ref{cor:Consistency-variable-selection} are reasonable and can be shown to hold under various dependence settings; see~\cite{Uniform:Kuch18}. Under these conditions, we get that $\hat{\beta}_{\hat{M}}$ ``converges'' to $\beta_{\hat{M}}$ and hence under reasonable conditions, it is only possible to perform consistent asymptotic inference only for $\beta_{\hat{M}}$ based on $\hat{\beta}_{\hat{M}}$. In other words, if a confidence region is constructed for a parameter $\eta$ centered at $\hat{\beta}_{\hat{M}}$ and that such region becomes a singleton asymptotically then $\|\eta - \beta_{\hat{M}}\|$ should converge to zero. In relation to the well-known consistent model selection literature, we can say if a claim is made about inference for $\beta_{M_0}$ (for $M_0$ the true support) then $\|\beta_{\hat{M}} - \beta_{M_0}\|$ should converge to zero asymptotically.
\subsection{Normal Approximation: Berry--Esseen result} 
From Corollary~\ref{cor:Uniform-in-Submodel} (if $\mathcal{D}_{\hat{M}}^{\Sigma} \to 0$), we have
\begin{equation}\label{eq:Linear-Approx-Variable-Selection}
\hat{\beta}_{\hat{M}} - \beta_{\hat{M}} ~\approx~ \Sigma_{\hat{M}}^{-1}(\hat{\Gamma}_{\hat{M}} - \hat{\Sigma}_{\hat{M}}\beta_{\hat{M}}),
\end{equation}
and hence inference for $\beta_{\hat{M}}$ requires understanding the asymptotic distribution of $\Sigma_{\hat{M}}^{-1}(\hat{\Gamma}_{\hat{M}} - \hat{\Sigma}_{\hat{M}}\beta_{\hat{M}})$ which is an average indexed by a random model $\hat{M}$. The impossibility results of Leeb and P{\"o}tscher~\citep{leeb2008can} imply that one cannot (uniformly) consistently estimate the asymptotic distribution of the right hand side of~\eqref{eq:Linear-Approx-Variable-Selection}. Hence the approach we take for inference is as follows: if we know apriori that $\hat{M}$ belongs on $\mathcal{M}$ either with probability $1$ or with probability approaching 1, then by simultaneously inferring about $\beta_M$ over all $M\in\mathcal{M}$ we can perform inference about $\beta_{\hat{M}}$. This is necessarily a conservative approach for any particular variable selection procedure leading to ($\hat{M}$ or) $\beta_{\hat{M}}$ but over all random models $\hat{M}\in\mathcal{M}$, this procedure is exact (or non-conservative); see~\citet[Theorem 3.1]{kuchibhotla2018valid}. We acheive this simultaneous inference by using high-dimensional normal approximation results for averages of random vectors. Based on Corollary~\ref{cor:Uniform-in-Submodel}, we prove the following corollary (similar to Corollary~\ref{cor:BEOLS}). 

Because of the finite sample nature (not requiring any specific structure), the result is cumbersome and requires some notation. We first briefly describe the method of proof of corollary to make the notation and result clear. We have already proved~\eqref{eq:Inf-Expansion-sub-model} for all $M\in\mathcal{M}$. Since Euclidean norm majorizes the maximum norm,
\[
\max_{1\le j\le |M|}|(\hat{\beta}_M - \beta_M)_j - (\Sigma_M^{-1}(\hat{\Gamma}_M - \hat{\Sigma}_M\beta_M))_j| \lesssim \frac{\mathcal{D}_M^{\Sigma}\|\Sigma_M^{-1}(\hat{\Gamma}_M - \hat{\Sigma}_M\beta_M)\|_{\Sigma_M}}{(1 - \mathcal{D}_M^{\Sigma})_+}. 
\]
Here we write $\lesssim$ since scaled Euclidean norm leads to other constant factors. We can use CLT for $(\Sigma_M^{-1}(\hat{\Gamma}_M - \hat{\Sigma}_M\beta_M))_{M\in\mathcal{M}}$ to compare $(\hat{\beta}_M - \beta_M)_{M\in\mathcal{M}}$ to a Gaussian counterpart. The CLT error term for the averages $(\Sigma_M^{-1}(\hat{\Gamma}_M - \hat{\Sigma}_M\beta_M))_{M\in\mathcal{M}}$ is defined as $\Delta_{n,\mathcal{M}}$. Here we also note that $\hat{\beta}_M - \beta_M$ is only close to the average upto an error term on the right hand side. This leads to two terms: first we need to show the right hand side term is indeed small for which we use CLT for scaled Euclidean norm (leading to $\Xi_{n,\mathcal{M}}$ below) and secondly, we need to account for closeness upto this small error which appears as probability of Gaussian process belonging in a small strip (leading to an anti-concentration term in the bound).

Now some notation. Let $V_M$ represent the version of $V$ in~\eqref{eq:V-definition} for model $M$,
\begin{equation}
\textstyle
V_M ~:=~ \mbox{Var}\left(n^{-1}\sum_{i=1}^n X_{i,M}(Y_i - X_{i,M}^{\top}\beta_M)\right).
\end{equation}
Note that $V_M = O(n^{-1})$, in general. Define the Gaussian process $(G_{M,j})_{M\in\mathcal{M},1\le j\le |M|}$ with mean zero and the covariance operator given by: $\mbox{Cov}(G_{M,j}, G_{M',j'})$ equals
\begin{equation}
\mbox{Cov}\left(\frac{1}{n}\sum_{i=1}^n \frac{(\Sigma_M^{-1}X_{i,M})_j(Y_i - X_{i,M}^{\top}\beta_M)}{{(\Sigma_M^{-1}V_M\Sigma_M^{-1})_j}^{1/2}},\, \frac{1}{n}\sum_{i=1}^n \frac{(\Sigma_{M'}^{-1}X_{i,{M'}})_{j'}(Y_i - X_{i,M'}^{\top}\beta_{M'})}{{(\Sigma_{M'}^{-1}V_{M'}\Sigma_{M'}^{-1})_{j'}}^{1/2}}\right),
\end{equation}
for all $M,M'\in\mathcal{M}$ and $1\le j\le |M|, 1\le j'\le |M'|$. Note $(G_{M,j})$ depends on $n$ but the marginal variances are all $1$. Let $\pi_s$ for $1\le s \le d$ represent the proportion of models of size $s$ in $\mathcal{M}$, that is, $\pi_s := \#\{M\in\mathcal{M}:|M| = s\}/|\mathcal{M}|.$
Now set $D := \sum_{M\in\mathcal{M}} 5^{|M|}$ and define
\[
\Xi_{n,\mathcal{M}} := \sup_{a\in\mathbb{R}_+^{D}}\,\left|\mathbb{P}\left(\left(\theta^{\top}V_M^{-\frac{1}{2}}(\hat{\Gamma}_M - \hat{\Sigma}_M\beta_M)\right)_{\substack{M\in\mathcal{M},\\\theta\in\mathcal{N}_{|M|}^{1/2}}} \preceq a\right) - \mathbb{P}\left(\left(\theta^{\top}\bar{G}_M\right)_{\substack{M\in\mathcal{M},\\\theta\in\mathcal{N}_{|M|}^{1/2}}}\preceq a\right)\right|,
\]
where $\preceq$ represents the vector coordinate-wise inequality, $\mathcal{N}_{|M|}^{1/2}$ represents the $1/2$-net of $\{\theta\in\mathbb{R}^{|M|}:\,\|\theta\| \le 1\}$, that is, $\min_{\theta'\in\mathcal{N}_{|M|}^{1/2}}\max_{\theta\in\mathbb{R}^{|M|}:\,\|\theta\| = 1}\,\|\theta - \theta'\| \le 1/2,$
and $(\bar{G}_M)_{M\in\mathcal{M}}$ represents a Guassian process that has mean zero and shares the same covariance structure as $(V_M^{-1/2}(\hat{\Gamma}_M - \hat{\Sigma}_M\beta_M))_{M\in\mathcal{M}}$. Note that $\mbox{Var}(G_M) = I_{|M|}$ for any $M\in\mathcal{M}$. The quantity~$\Xi_{n,\mathcal{M}}$ helps control one of the remainder factors, $\|\Sigma_M^{-1}(\hat{\Gamma}_M - \hat{\Sigma}_M\beta_M)\|_{\Sigma_MV_M^{-1}\Sigma_M}$. For the main term, define $C := \sum_{M\in\mathcal{M}}|M|$ and
\[\textstyle
\Delta_{n,\mathcal{M}} := \sup_{a\in\mathbb{R}_+^{C}}\,\left|\mathbb{P}\left(\left(\frac{|(\Sigma_M^{-1}(\hat{\Gamma}_M - \hat{\Sigma}_M\beta_M))_j|}{{(\Sigma_M^{-1}V_M\Sigma_M^{-1})_j}^{1/2}}\right)_{\substack{M\in\mathcal{M},\\1\le j\le |M|}} \preceq a\right) - \mathbb{P}\left((|G_{M,j}|)_{\substack{M\in\mathcal{M},\\1\le j\le |M|}} \preceq a\right)\right|.
\]
\begin{cor}\label{cor:BEOLS-Varaible-Selection}
For all $M\subseteq\{1,2,\ldots,d\}$, we have
\[
\|\hat{\beta}_M - \beta_M - \Sigma_M^{-1}(\hat{\Gamma}_M - \hat{\Sigma}_M\beta_M)\|_{\Sigma_MV_M^{-1}\Sigma_M} \le \frac{\mathcal{D}^{\Sigma}_M\|\Sigma_M^{-1}(\hat{\Gamma}_M - \hat{\Sigma}_M\beta_M)\|_{\Sigma_MV_M^{-1}\Sigma_M}}{(1 - \mathcal{D}_M^{\Sigma})_+}.
\]
Furthermore, for any $(\eta_M)_{M\in\mathcal{M}} \preceq 1/2$, we have
\begin{align}\label{eq:Normal-Approx-Variable-Selection}
\begin{split}
&\sup_{a\in\mathbb{R}_+^{C}}\,\left|\mathbb{P}\left(\left(\frac{|(\hat{\beta}_M - \beta_M)_j|}{{(\Sigma_M^{-1}V_M\Sigma_M^{-1})_j}^{1/2}}\right)_{\substack{M\in\mathcal{M},\\1\le j\le |M|}} \preceq a\right) - \mathbb{P}\left((|G_{M,j}|)_{\substack{M\in\mathcal{M},\\1\le j\le |M|}} \preceq a\right)\right|\\
&\quad\le \Delta_{n,\mathcal{M}} + 2.65\Xi_{n,M} + \mathbb{P}\left(\max_{M\in\mathcal{M}}{\mathcal{D}_M^{\Sigma}}/{\eta_M} \ge 1\right)\\
&\quad\quad+ \sup_{a\in\mathbb{R}_+^{C}}\mathbb{P}\left(\bigcup_{\substack{M\in\mathcal{M},\\1\le j\le |M|}}\left\{||G_{M,j}| - a_{M,j}| \le {4\eta_M\sqrt{2\log\left(\frac{|\mathcal{M}|\pi_{|M|}5^{2|M|}}{\Xi_{n,M}}\right)}}\right\}\right).
\end{split}
\end{align}
\end{cor}
The proof of Corollary~\ref{cor:BEOLS-Varaible-Selection} can be found in Appendix~\ref{AppSec:ProofBEOLS-Variable-Selection}. The first inequality in Corollary~\ref{cor:BEOLS-Varaible-Selection} is slightly different from the conclusion of Corollary~\ref{cor:Uniform-in-Submodel} but is more important for inference since the scaling in Corollary~\ref{cor:BEOLS-Varaible-Selection} is with respect to the ``asymptotic'' variance of $\hat{\beta}_M - \beta_M$. The second conclusion of Corollary~\ref{cor:BEOLS-Varaible-Selection} is a ``randomness-free'' version of finite sample Berry--Esseen type result for $(\hat{\beta}_M - \beta_M)$ simultaneously over all $M\in\mathcal{M}$. The terms each have a meaning and is explained before the notation above. For a simpler result, consider the case of fixed (non-stochastic) covariates. In this case $\mathcal{D}_M^{\Sigma} = 0$ for all $M$ and hence the result becomes
\[
\left|\mathbb{P}\left(\left(\frac{|(\hat{\beta}_M - \beta_M)_j|}{{(\Sigma_M^{-1}V_M\Sigma_M^{-1})_j}^{1/2}}\right)_{\substack{M\in\mathcal{M},\\1\le j\le |M|}} \preceq a\right) - \mathbb{P}\left((|G_{M,j}|)_{\substack{M\in\mathcal{M},\\1\le j\le |M|}} \preceq a\right)\right| \le \Delta_{n,\mathcal{M}} + 3\Xi_{n,\mathcal{M}},
\]
for all $a\in\mathbb{R}_+^C$ since we can take $\eta_M$ to be zero in limit. Getting back to the bound in Corollary~\ref{cor:BEOLS-Varaible-Selection}, the quantities $\Delta_{n,\mathcal{M}}$ and $\Xi_{n,\mathcal{M}}$ can be easily controlled by using high-dimensional CLT results which only depend on the number of coordinates in the vector logarithmically. In particular for $\max\{\Delta_{n,\mathcal{M}}, \Xi_{n,\mathcal{M}}\} = o(1)$ they only require $\log(\sum_{M\in\mathcal{M}}|M|) = o(n^{\gamma})$ for some $\gamma > 0$~\citep{Chern17,Chern14,ZhangWu17,Zhang14,koike2019high} for details. For instance, if $\mathcal{M}=\{M\subseteq\{1,\ldots,d\}:\,|M| \le k\}$ then the requirement becomes $k\log(ed/k) = o(n^{\gamma}).$ For the case of independent observations and sufficiently weakly dependent observations, we have
\[
\max\{\Delta_{n,\mathcal{M}}, \Xi_{n,\mathcal{M}}\} ~=~ O(1)\left(n^{-1}{\log^7\textstyle(\sum_{M\in\mathcal{M}} |M|)}\right)^{1/6}. 
\]
Bounds for $\mathbb{P}(\cup_{M\in\mathcal{M}}\{\mathcal{D}_M^{\Sigma} \ge \eta_M\})$ can be obtained using certain tail and ``weak dependence'' assumptions the covariates $X_1,\ldots,X_n$ (and as mentioned before one only needs to be concerned with the stochastic coordinates of covariates). This often necessitates exponential tails on the covariates if the total number of covariates~$d$ is allowed to grow almost exponentially with $n$~\citep{guedon2015interval,tikhomirov2017sample}. Finally the control of the anti-concentration term (the last one in Corollary~\ref{cor:BEOLS-Varaible-Selection}) only concerns a tail properties of a Gaussian process. A dimension dependent bound (that only depends logarithmically on dimension) for this probability can be found in~\citep{Naz03,chernozhukov2017detailed}:
\[\textstyle
\mathbb{P}\left(\bigcup_{{M\in\mathcal{M},1\le j\le |M|}}\left\{||G_{M,j}| - a_{M,j}| \le \varepsilon\right\}\right) \le H\varepsilon\sqrt{\log\left(\sum_{M\in\mathcal{M}}|M|\right)},
\]
for some constant $H > 0$. Dimension-free bounds for this probability exist only for some special cases~\citep{Chern15,2018arXiv180606153K}. Regarding the constant in the anti-concentration probability, note that $\pi_{|M|}|\mathcal{M}| \le (ed/|M|)^{|M|}$ for any collection $\mathcal{M}$and hence $\log(|M|\pi_{|M|}5^{2|M|}/\Xi_{n,\mathcal{M}}) \le |M|\log(25ed/\{|M|\Xi_{n,\mathcal{M}}\})$.
\subsection{Inference under Variable Selection}\label{subsec:Inference-OLS-Variable-Selection} Suppose that we can find $(\eta_M)_{M\in\mathcal{M}}$ such that $\mathbb{P}(\cup_{M\in\mathcal{M}}\{\mathcal{D}_M^{\Sigma} \ge \eta_M\})$ and the anti-concentration term goes to zero, then from Corollary~\ref{cor:BEOLS-Varaible-Selection} we get that
\[
\mathbb{P}\left(\left((\Sigma_M^{-1}V_M\Sigma_M^{-1})_j^{-1/2}{|(\hat{\beta}_M - \beta_M)_j|}\right)_{\substack{M\in\mathcal{M},\\1\le j\le |M|}} \preceq a\right) ~\approx~ \mathbb{P}\left((|G_{M,j}|)_{\substack{M\in\mathcal{M},\\1\le j\le |M|}} \preceq a\right),
\] 
uniformly for all $a\in\mathbb{R}^{\sum_{M\in\mathcal{M}}|M|}$. In order to perform inference (or in particular confidence regions) one can choose a vector $a = a_{\alpha}$ such that
\begin{equation}\label{eq:Quantile-Gaussian-Process}
\mathbb{P}\left((|G_{M,j}|)_{\substack{M\in\mathcal{M},\\1\le j\le |M|}} \preceq a_{\alpha}\right) = 1 - \alpha.
\end{equation}
This implies that for any $\hat{M}\in\mathcal{M}$ chosen (possibly) randomly based on the data,
\[
\mathbb{P}\left(\left({(\Sigma_{\hat{M}}^{-1}V_{\hat{M}}\Sigma_{\hat{M}}^{-1})_j}^{-1/2}{|(\hat{\beta}_{\hat{M}} - \beta_{\hat{M}}|)_j}\right)_{1\le j\le |\hat{M}|} \le (a_{\alpha})_{\hat{M}}\right) \ge 1 - \alpha + o(1),
\]
asymptotically. This means that with (asymptotic) probability of at least $1 - \alpha$, $\beta_{\hat{M},j}$ belongs in the interval $[\hat{\beta}_{\hat{M},j} \pm (a_{\alpha})_{\hat{M},j}{(\Sigma_{\hat{M}}^{-1}V_{\hat{M}}\Sigma_{\hat{M}}^{-1})_j}^{1/2}]$ simultaneously for all $1\le j\le |\hat{M}|$. If no variable selection is involved and no simultaneity over $1\le j\le |\hat{M}|$ is required, then $(a_{\alpha})_{\hat{M},j}$ would just be $z_{\alpha/2}$ (the usual normal quantile for a $(1-\alpha)$-confidence interval). This is the essential point of post-selection inference wherein we enlarge the usual confidence intervals to make them simultaneous. 

The above discussion completes inference for the OLS estimator under variable selection for all types of observations (that allow for a CLT: $\Delta_{n,\mathcal{M}}\asymp\Xi_{n,\mathcal{M}}\asymp 0$) except for two important points: firstly, we have proved the CLT result with the true ``asymptotic'' variance $\Sigma_M^{-1}V_M\Sigma_M^{-1}$ which is unknown in general; it is, however, easy to estimate this variance using the techniques described in Section~\ref{sec:Statistical-Inference}. Secondly and more importantly, there are infinitely many different choices of $a_{\alpha}$ satisfying~\eqref{eq:Quantile-Gaussian-Process}; what is the right choice? The first problem is easy to rectify in that if a variance estimator $\hat{\sigma}_{M,j}$ (for ${(\Sigma_M^{-1}V_M\Sigma_M^{-1})_j}^{1/2}$) has a good enough rate of convergence with respect to the metric $|\hat{\sigma}_{M,j}/{(\Sigma_M^{-1}V_M\Sigma_M^{-1})_j}^{1/2} - 1|$ uniformly over all $M\in\mathcal{M}, 1\le j\le |M|$ then it is easy to prove a version of Corollary~\ref{cor:BEOLS-Varaible-Selection} with the unknown variance replaced by the estimator in the first probability.  

Related to the choice of $(a_{\alpha})_{M\in\mathcal{M},1\le j\le |M|}$, in the path-breaking work~\cite{Berk13}, the authors have used $(a_{\alpha}) = a\mathbf{1}$ for some constant $a$, which means that the simultaneous inference is based on quantiles of the maximum statistic:
\begin{equation}\label{eq:Maximum-Normal-Approximation}
\max_{M\in\mathcal{M}}\max_{1\le j\le |M|}{(\Sigma_M^{-1}V_M\Sigma_M^{-1})_j}^{-1/2}{|(\hat{\beta}_M - \beta_M)_j|}.
\end{equation}
\cite{Berk13} assumed non-stochastic covariates and an independent homoscedsatic Gaussian model for the response. This statistic was also adopted in~\cite{Bac16} where the framework was generalized to the case of non-Gaussian responses (but with non-stochastic covariates); further both works require the total number of covariates to be fixed and not change with $n$. The analysis above does not require either of these conditions since our results are \emph{deterministic}. 
Hence
\begin{equation}\label{eq:Max-t-quantile}
\mathbb{P}\left(\max_{{M\in\mathcal{M},\\1\le j\le |M|}}\,|G_{M,j}| \le K(\alpha)\right) = 1 - \alpha,
\end{equation}
implies for any $\hat{M}$ such that $\mathbb{P}(\hat{M}\in\mathcal{M}) = 1$, we have asymptotically
\begin{equation}\label{eq:max-t-model-quantile}
\mathbb{P}\left(\max_{1\le j\le |\hat{M}|}\left|{(\Sigma_{\hat{M}}^{-1}V_{\hat{M}}\Sigma_{\hat{M}}^{-1})_j}^{-1/2}{(\hat{\beta}_{\hat{M}} - \beta_{\hat{M}})_j}\right| \le K(\alpha)\right) \ge 1 - \alpha.
\end{equation}
The quantile $K(\alpha)$ in~\eqref{eq:Max-t-quantile} can be computed by bootstrapping the maximum statistic using the linear representation result; see~\cite{belloni2018high}, \cite{deng2017beyond} and~\cite{Zhang14} for details on bootstrap for independent/dependent summands in averages.

The maximum statistic in~\eqref{eq:Maximum-Normal-Approximation} (used in~\cite{Berk13} and~\cite{Bac16}) is only one of the many different ways of performing valid post-selection inference. It is clear that if for some $\alpha\in[0, 1]$ and numbers $\{K_M(\alpha):\,M\in\mathcal{M}\}$,
\begin{equation}\label{eq:Rectangle-Union-PoSI}
\mathbb{P}\left(\bigcap_{M\in\mathcal{M}}\left\{\max_{1\le j\le |M|}\left|G_{M,j}\right| \le K_M(\alpha)\right\}\right) = 1 - \alpha,
\end{equation}
then we have
\begin{equation}\label{eq:Rectangle-PoSI}
\mathbb{P}\left(\max_{1\le j\le |\hat{M}|}\left|{(\Sigma_{\hat{M}}^{-1}V_{\hat{M}}\Sigma_{\hat{M}}^{-1})_j}^{-1/2}{(\hat{\beta}_{\hat{M}} - \beta_{\hat{M}})_j}\right| \le K_{\hat{M}}(\alpha)\right) \ge 1 - \alpha + o(1),
\end{equation}
for any $\hat{M}$ (possibly random) such that $\mathbb{P}(\hat{M}\in\mathcal{M}) = 1$ (this equality can be relaxed to convergence to 1). Inequality~\eqref{eq:Rectangle-PoSI} readily implies (asymptotically valid) post-selection confidence region for $\beta_{\hat{M}}$ as
\[
\hat{\mathcal{R}}_{\infty,\hat{M}} := \left\{\theta\in\mathbb{R}^{|\hat{M}|}:\,\max_{1\le j\le |\hat{M}|}\left|{(\Sigma_{\hat{M}}^{-1}V_{\hat{M}}\Sigma_{\hat{M}}^{-1})_j^{-1/2}}{(\hat{\beta}_{\hat{M},j} - \theta_j)}\right| \le K_{\hat{M}}(\alpha)\right\}.
\]
Note that the confidence regions or more generally inference obtained from the maximum statistic corresponds to taking $(K_M(\alpha))_{M\in\mathcal{M}}$ in~\eqref{eq:Rectangle-Union-PoSI} to be a constant multiple of $(1)_{M\in\mathcal{M}}$ (all $1$'s vector). Further note that the event in~\eqref{eq:Rectangle-Union-PoSI} represents a specific choice of vector $a_{\alpha}$ in~\eqref{eq:Quantile-Gaussian-Process} for which Corollary~\ref{cor:BEOLS-Varaible-Selection} applies. Before we discuss how to choose $(K_M(\alpha))_{M\in\mathcal{M}}$, we list out some of the disadvantages of using the maximum statistic~\eqref{eq:Maximum-Normal-Approximation}. 
\paragraph{Disadvantages of the maximum statistic.} The maximum statistic is a natural generalization of inference for a single model to simultaneous inference over a collection of models. The maximum statistic would be the right thing to do if we are concerned with simultaneous inference for $p$ parameters (all of which are of same order) but this is not the case with OLS under variable selection. It is intuitively expected that models with more number of covariates would have larger width intervals. For this reason by taking the maximum over the collection $\mathcal{M}$ of models, one is ignoring the smaller models and the fact that small models have smaller width confidence intervals. To be concrete, if $\mathcal{M}$ is $\mathcal{M}_{\le k}$ it follows from the results of~\cite{Berk13,7915760} that
\begin{equation}\label{eq:Maximum-worst-case}
\max_{M\in\mathcal{M}}\max_{1\le j\le |M|}\,|G_{M,j}| = O_p(\sqrt{k\log(ed/k)}),
\end{equation}
and in the worst case this rate can be attained. But if $k = 40$ (for example) but the selected model $\hat{M}$ happened to have only two covariates, then the confidence interval is (unnecessarily) wider by a factor of $\sqrt{20}$. By allowing model dependent quantile $K_M(\alpha)$ as in~\eqref{eq:Rectangle-Union-PoSI} we can tighten confidence intervals appropriately. For this particular disadvantage, it is enough to have $K_M(\alpha)$ depend on $M$ only through $|M|$, its size. There is a second disadvantage of the maximum statistic that requires dependence of $K_M(\alpha)$ on the covariates in $M$.

To describe the second disadvantage we look at the conditions under which worst case rate in~\eqref{eq:Maximum-worst-case} is attained when $k = d$. \citet[Section 6.2]{Berk13} shows that if the covariates are non-stochastic, and
\[
\hat{\Sigma} := \begin{bmatrix}I_{d-1} & c\mathbf{1}_{d-1}\\\mathbf{0}_{d-1}^{\top} & \sqrt{1 - (d-1)c^2}\end{bmatrix},\mbox{ for some $c^2 < 1/(d-1)$},
\]
then there exists a constant $\mathfrak{C} > 0$, such that 
\begin{equation}\label{eq:Max-t-worst-case}\max_{M\in\mathcal{M}_{\le d}}\,\max_{1\le j\le |M|}|G_{M,j}| \ge \mathfrak{C}\sqrt{d}.\end{equation} Now define $\mathcal{M} = \{M\subseteq\{1,\ldots,d\}:M\subseteq\{1,\ldots,d-1\}\}$, that is, $\mathcal{M}$ is the collection of models that only contain the first $d-1$ covariates. It now follows from~\cite[Section 6.1]{Berk13} that
\begin{equation}\label{eq:Max-t-best-case}
\max_{M\in\mathcal{M}}\max_{1\le j\le |M|}|G_{M,j}| \asymp \sqrt{\log(ed)}.
\end{equation}
Comparing~\eqref{eq:Max-t-worst-case} and~\eqref{eq:Max-t-best-case}, it is clear that the inclusion of the last covariate increases the order of the maximum statistic from $\sqrt{\log(ed)}$ to $\sqrt{d}$; this shift is because of increased collinearity. This means that if in the selection procedure we allow all models but end up choosing the model that only contains the first $d-1$ covariates, we pay of lot more price than necessary. Note that if $d$ increases with $n$, this increase (in rate) could hurt more. Once again allowing for $K_{M}(\alpha)$ a model dependent quantile for maximum (over $j$) in that model resolves this disadvantage.
\paragraph{How to choose $K_M(\alpha)$?} Now that we have understood the need for model $M$ dependent quantiles $K_M(\alpha)$, it remains to decide how to find these quantiles. But first note that these are not uniquely defined because multivariate quantiles are not unique. We do not yet know of an ``optimal'' construction of $K_M(\alpha)$ and we describe a few choices below motivated by multi-scale testing literature~\citep{dumbgen2001multiscale,datta2018optimal}. Before we proceed to this, we note an impossibility on uniform improvement over the maximum statistic. Suppose we select a (random) model $\hat{M}$ such that
\[
\max_{1\le j\le |\hat{M}|}\left|{(\hat{\beta}_{\hat{M},j} - \beta_{\hat{M},j})}/{\sigma_{\hat{M},j}}\right| ~=~ \max_{M\in\mathcal{M}}\max_{1\le j\le |M|}\left|{(\hat{\beta}_{M,j} - \beta_{M,j})}/{\sigma_{M,j}}\right|,
\]
where $\sigma_{M,j}$ represents the standard deviation, $(\Sigma_M^{-1}V_M\Sigma_M^{-1})_j^{1/2}$, of $\hat{\beta}_{M,j} - \beta_{M,j}$. For this random model $\hat{M}$, $K(\alpha)$ the quantile of the maximum statistic in~\eqref{eq:Max-t-quantile} leads to the smallest possible rectangular confidence region for $\beta_{\hat{M}}$. This implies that $K_{\hat{M}}(\alpha) \ge K(\alpha)$ for any $\alpha\in[0, 1]$ and any sequence $(K_M(\alpha))_{M\in\mathcal{M}}$. Therefore no sequence of quantiles $(K_M(\alpha))_{M\in\mathcal{M}}$ satisfying~\eqref{eq:Rectangle-Union-PoSI} can improve on $K(\alpha)$ uniformly over $M\in\mathcal{M}$; any gain for some model is paid for by a loss for some other model. The hope is that the gain outweighs the loss and we see this in our simulations. 

Getting back to the construction of $K_M(\alpha)$, let the maximum for model $M$ be
\[
T_M := \max_{1\le j\le |M|}\,\left|{(\hat{\beta}_{M,j} - \beta_{M,j})}/{\hat{\sigma}_{M,j}}\right|,
\]
for an estimator $\hat{\sigma}_{M,j}$ of the standard deviation $\sigma_{M,j}$; recall $\sigma_{M,j}$ involves $V_M$ that converges to zero. Recall that the maximum statistic~\eqref{eq:Maximum-Normal-Approximation} is given by $\max_{M\in\mathcal{M}} T_M$. We now present three choices that will lead to three different quantiles $K_M(\alpha)$. 
\begin{enumerate}\label{page:Different-ways-posi}
  \item In order to take into account the set of covariates in $M$, we center $T_M$ by its median before taking the maximum:
  \begin{equation}\label{eq:Max-with-centering-no-scaling}
  \max_{M\in\mathcal{M}}\,\left\{T_M - \texttt{med}(T_M)\right\},
  \end{equation}
  where $\texttt{med}(\cdot)$ represents the median.
  One can center by the mean of $T_M$ but estimation of mean of a maximum using bootstrap is not yet clear. 
  Higher collinearity between the covariates in $M$ could increase the order of $T_M$, the effect of which we avoid spilling into other models by centering by the median. Also, it is clear that the median of $T_M$ has order depending only on $M$ not the maximum model size in collection $\mathcal{M}$. Further it is well-known that the maximum of Gaussians exhibit a super-concentration phenomenon in that their variance decreases to zero as the number of entries in the maximum goes to infinity. For this reason, it may not be of importance to scale by the standard deviation of $T_M$. If $K_{\mathcal{M}}^{(1)}(\alpha)$ represents the quantile of the statistic~\eqref{eq:Max-with-centering-no-scaling}, then the post-selection confidence intervals are given by
  \[
  \hat{\mathcal{R}}_M^{(1)} := \left\{\theta\in\mathbb{R}^{|M|}:\,\max_{1\le j\le |M|}|{(\hat{\beta}_{M,j} - \theta_j)}/{\hat{\sigma}_{M,j}}| \le \widehat{\texttt{med}}(T_M) + K_{\mathcal{M}}^{(1)}(\alpha)\right\}.
  \]
  \item The super-concentration of the maximum of Gaussians holds only under certain ``strong uncorrelatedness'' assumption. Following the previous suggestion, we can normalize the centered $T_M$ by its median absolute deviation (MAD) to account for the variance:
  \begin{equation}\label{eq:Max-with-centering-scaling}
  \max_{M\in\mathcal{M}}\,{\{T_M - \texttt{med}(T_M)\}}/{\texttt{MAD}(T_M)},
  \end{equation}
  where $\texttt{MAD}(T_M) := \texttt{med}(|T_M - \texttt{med}(T_M)|)$. If $K_{\mathcal{M}}^{(2)}(\alpha)$ represents the quantile of the statistic~\eqref{eq:Max-with-centering-scaling}, then the post-selection confidence intervals are given by
  \[
  \hat{\mathcal{R}}_M^{(2)} := \left\{\theta\in\mathbb{R}^{|M|}:\,\max_{1\le j\le |M|}|{(\hat{\beta}_{M,j} - \theta_j)}/{\hat{\sigma}_{M,j}}| \le \widehat{\texttt{med}}(T_M) + \widehat{\texttt{MAD}}(T_M)K_{\mathcal{M}}^{(2)}(\alpha)\right\}.
  \]
  \item Now that we have centered and scaled $T_M$ with its median and MAD, it is expected that even for models of different sizes, $(T_M - \texttt{med}(T_M))/\texttt{MAD}(T_M)$ are of the same order. However, when we take the maximum over all models of same size they may not be. The reason for this is the maximum over models of size $1$ involves $d$ terms and the maximum over models of size $2$ involves $d(d-1)/{2}$ terms. Hence naturally the maximum over models of size 2 is expected to be bigger. To account for this discrepancy define the centered and scaled maximum statistic for model size $s$ as
  \begin{equation}
  \textstyle
  \mathfrak{T}_s := \max_{|M| = s}{\{T_M - \texttt{med}(T_M)\}}/{\texttt{MAD}(T_M)},
  \end{equation}
  and take quantile of
  \begin{equation}\label{eq:centering-scaling-and-centering}
  \max_{1\le s\le k}\,\{\mathfrak{T}_s - \texttt{med}(\mathfrak{T}_s)\}.
  \end{equation}
  If $K_{\mathcal{M}}^{(3)}(\alpha)$ represents the quantile of the statistic~\eqref{eq:centering-scaling-and-centering}, then the post-selection confidence intervals are given by
  \[
  \hat{\mathcal{R}}_M^{(3)} := \left\{\theta:\max_{1\le j\le |M|}\left|\frac{\hat{\beta}_{M,j} - \theta_j}{\hat{\sigma}_{M,j}}\right| \le \widehat{\texttt{med}}(T_M) + \widehat{\texttt{MAD}}(T_M)[K_{\mathcal{M}}^{(3)}(\alpha) + \widehat{\texttt{med}}(\mathfrak{T}_{|M|})]\right\}.
  \]
\end{enumerate}
We emphasize once again that even though these choices improve the width of confidence intervals for some models, they will deteriorate the width for other models. We will see from the simulations in Section~\ref{sec:Simulations} that the gain (for some models) outweighs the loss (for other models) in width. All the choices above involve $\texttt{med}(T_M)$, $\texttt{MAD}(T_M)$ which are simple functions of quantiles and can be computed readily from bootstrap procedures mentioned above.
\section{Rates under Independence}\label{sec:Independence}
All the theoretical analysis in previous sections is deterministic and the complete study in any specific setting requires bounding the remainder terms in the deterministic inequalities above. In this section, we complete the program by bounding the remainder terms in case of independent observations. The two main quantities that need bounding for Theorem~\ref{thm:Deterministic-Ineq} are 
\[
\mathcal{D}^{\Sigma} := \|\Sigma^{-1/2}\hat{\Sigma}\Sigma^{-1/2} - I_d\|_{op}\quad\mbox{and}\quad \|\Sigma^{-1}(\hat{\Gamma} - \hat{\Sigma}\beta)\|_{\Sigma} = \|\Sigma^{-1/2}(\hat{\Gamma} - \hat{\Sigma}\beta)\|.
\]
The concentration of the sample covariance matrix to its expectation has been the study for decades documented in the works of~\cite{Ver12,Vershynin18},~\cite{Rud13},~\cite{guedon2015interval},~\cite{tikhomirov2017sample}. We state here the result from~\cite{tikhomirov2017sample} with minimal tail assumptions that we know of.
\begin{thm}[Theorem 1.1 of~\cite{tikhomirov2017sample}]
Fix $n \ge 2d$ and $p \ge 2$. If $X_1,\ldots,X_n$ are centered iid random vectors satisfying: for some $B \ge 1$,
\begin{equation}\label{eq:Polynomial-moment}\textstyle
\mathbb{E}|a^{\top}\Sigma^{-1/2}X|^p \le B^p\quad\mbox{for all}\quad a\in\mathbb{R}^d,\mbox{ with }\|a\| = 1.
\end{equation}
Then there exists a constant $K_p > 0$ with probability at least $1 - 1/n$,
\[
\mathcal{D}^{\Sigma} \le \frac{K_p}{n}\max_{1\le i\le n} \|\Sigma^{-1/2}X_i\|^2 + K_pB^{2}\left(\frac{d}{n}\right)^{1-2/p}\log^4\left(\frac{n}{d}\right) + K_pB^{2}\left(\frac{d}{n}\right)^{1 - 2/\min\{p,4\}}.
\]
\end{thm}
The random quantity on the right hand side can be bounded using appropriate bounds on $\mathbb{E}[\|\Sigma^{-1/2}X_i\|^{2q}/d^q]$ for some $q\ge1$. Assuming the first term can be ignored compared to the others, we get that $\mathcal{D}^{\Sigma}$ converges to zero as long as $d = o(n)$ when the covariates have at least $(2 + \delta)$-moments. Further if $p \ge 4$, then $\mathcal{D}^{\Sigma} = O_p(\sqrt{d/n})$. Regarding the term $\|\Sigma^{-1/2}(\hat{\Gamma} - \hat{\Sigma}\beta)\|$, we have
\begin{align*}
\mathbb{E}\|\Sigma^{-1/2}(\hat{\Gamma} - \hat{\Sigma}\beta)\| &\le \sqrt{\mbox{tr}(\mbox{Var}(\Sigma^{-1/2}(\hat{\Gamma} - \hat{\Sigma}\beta)))} = \frac{(\mathbb{E}[\|\Sigma^{-1/2}X\|^2(Y - X^{\top}\beta)^2])^{1/2}}{\sqrt{n}}.
\end{align*}
Hence if $\mathbb{E}[\|\Sigma^{1/2}X\|^2(Y - X^{\top}\beta)^2] = O(d)$, then we get $\|\Sigma^{-1/2}(\hat{\Gamma} - \hat{\Sigma}\beta)\| = O_p(\sqrt{d/n})$. Combining these calculations with Theorem~\ref{thm:Deterministic-Ineq}, we get
\[
\|\hat{\beta} - \beta\|_{\Sigma} = O_p(1)\sqrt{\frac{d}{n}}\quad\mbox{and}\quad \|\hat{\beta} - \beta - \Sigma^{-1}(\hat{\Gamma} - \hat{\Sigma}\beta)\|_{\Sigma} = O_p(1)\frac{d}{n},
\]
allowing for $d$ growing with the sample size $n$; consistency holds when $d = o(n)$ and asymptotic normality holds when $d = o(\sqrt{n})$. Asymptotic analysis for $d/n\to\kappa\in[0, 1)$ can be done with more stringent conditions on the observations.

Regarding Corollary~\ref{cor:Uniform-in-Submodel}, we need to control \emph{simultaneously} over $M\in\mathcal{M}$,
\begin{equation}\label{eq:Simultaneous-remainder-control}
\mathcal{D}_M^{\Sigma} = \|\Sigma_M^{-1/2}\hat{\Sigma}_M\Sigma_M^{-1/2} - I_{|M|}\|_{op}\quad\mbox{and}\quad \|\Sigma^{-1/2}_M(\hat{\Gamma}_M - \hat{\Sigma}_M\beta_M)\|.
\end{equation}
This simultaneous control often necessitates exponential tails for covariates if one needs to allow $d$ to grow (almost exponentially) with $n$. \cite{guedon2015interval} provide sharp results for $\sup_{|M|\le k}\|\hat{\Sigma}_M - \Sigma_M\|_{op}$ for both polynomial and exponential tails on covariates. We do not know such sharp results for $\sup_{M\in\mathcal{M}}\mathcal{D}_M^{\Sigma}$. By a simple union bound the following result can be proved for both quantities in~\eqref{eq:Simultaneous-remainder-control}. For this we assume the following extension of~\eqref{eq:Polynomial-moment}: For all $1\le i\le n$,
\begin{equation}\label{eq:Sub-Weibull-Covariance-Scaled}
\mathbb{E}\left[\exp\left(\frac{|a^{\top}X_i|^{\beta}}{\mathfrak{K}_{\beta}^{\beta}\|a\|_{\Sigma}^{\beta}}\right)\right] \le 2,\;\;\mbox{for some $\beta > 0$, $0 < \mathfrak{K}_{\beta} < \infty$ and for all $a\in\mathbb{R}^d$.}
\end{equation}
Condition~\eqref{eq:Sub-Weibull-Covariance-Scaled} is same as sub-Gaussianity if $\beta = 2$ and is same as sub-exponentiality if $\beta = 1$. With $\beta = \infty$, it becomes a boundedness condition. If $X_i$'s satisfy condition~\eqref{eq:Sub-Weibull-Covariance-Scaled} with $\beta < 1$ then their moment generating function may not exist but they still exhibit ``weak'' exponential tails. Additionally note that~\eqref{eq:Sub-Weibull-Covariance-Scaled} does not require $\Sigma$ to be invertible and it implies that for all $M\subseteq\{1,2,\ldots,d\}$,
\begin{equation}\label{eq:Sub-Weibull-sub-model-Covariance-Scaled}
\mathbb{E}\left[\exp\left({\mathfrak{K}_{\beta}^{-\beta}|a^{\top}\Sigma_M^{-1/2}X_{i,M}|^{\beta}}\right)\right] \le 2\quad\mbox{for all}\quad a\in\mathbb{R}^{|M|}\mbox{ such that }\|a\| = 1.
\end{equation}
Define the kurtosis and ``regression variance'' for model $M$ as
\[
\kappa_M^{\Sigma} := \max_{\theta\in\mathbb{R}^{|M|}}\,\frac{1}{n}\sum_{i=1}^n \frac{\mbox{Var}((X_{i,M}^{\top}\theta)^2)}{\|\Sigma_M^{1/2}\theta\|^4}\quad\mbox{and}\quad \mathfrak{V}_M := \max_{\theta\in\mathbb{R}^{|M|}}\,\frac{1}{n}\sum_{i=1}^n \mbox{Var}(\theta^{\top}\Sigma_M^{-1/2}X_{i,M}Y_i).
\]
Assume the observations $(X_1,Y_1),\ldots,(X_n,Y_n)$ are just independent.
\begin{prop}\label{prop:Rates-D_M-Gamma_M}
Fix any $t\ge0$. Under~\eqref{eq:Sub-Weibull-sub-model-Covariance-Scaled}, we have with probability at least $1 - 3e^{-t}$, simultaneously for any $1\le s\le d$, for any $M\subseteq\{1,\ldots,d\}$ with $|M| = s$,
\begin{equation}\label{eq:Tail-Bound-D_M-Simultaneous}
\mathcal{D}_M^{\Sigma} \le 14\sqrt{\frac{\kappa_M^{\Sigma}(t + {s\log(9e^2d/s)})}{n}} + \frac{C_{\beta}\mathfrak{K}_{\beta}^{2}(\log(2n))^{2/\beta}(t + s\log(9e^2d/s))^{\max\{1, 2/\beta\}}}{n}.
\end{equation}
If~\eqref{eq:Sub-Weibull-sub-model-Covariance-Scaled} and $\mathbb{E}[Y_i^r] \le K_{n,r}^r$ for some $r\ge2$ hold true, then with probability at least $1 - 3e^{-t} - t^{-r+1}$, for any $1\le s\le d$, for any model $M\subseteq\{1,\ldots,d\}$ with $|M| = s$,
\begin{align}
\|\Sigma_M^{-1/2}(\hat{\Gamma}_M - \hat{\Sigma}_M\beta_M)\| &\le 14\sqrt{\frac{\mathfrak{V}_M(t + s\log(5e^2d/s))}{n}} + \mathcal{D}_M^{\Sigma}\textstyle(\sum_{i=1}^n \mathbb{E}[Y_i^2]/n)^{1/2}\nonumber\\
&\qquad+ \frac{C_{\beta}K_{n,r}\mathfrak{K}_{\beta}(\log(2n))^{1/\beta}(t + s\log(5e^2d/s))^{\max\{1,1/\beta\}}}{n^{1 - 1/r}}\nonumber\\
&\qquad+ \frac{tC_{\beta, r}K_{n,r}\mathfrak{K}_{\beta}(s\log(5e^2d/s) + \log n)^{1/\beta}}{n^{1 - 1/r}},\label{eq:influence-function-expansion-bound}
\end{align}
for some constants $C_{\beta}, C_{\beta, r} > 0$ depending only on $\beta$ and $(\beta, r)$, respectively.
\end{prop}
The proof of Proposition~\ref{prop:Rates-D_M-Gamma_M} can be found in Appendix~\ref{AppSec:Proof-of-prop-Rates-D_M-Gamma_M}. Note that the rates for $\mathcal{D}_M^{\Sigma}$ and for $\|\Sigma_M^{-1/2}(\hat{\Gamma}_M - \hat{\Sigma}_M\beta_M)\|$ scale with $|M|$ (and only logarithmically on the total number of covariates $d$) and we did not just bound $\max_{|M|\le k}\mathcal{D}_M^{\Sigma}$. This is what we tried to replicate in a data-driven way from the post-selection confidence regions in Page~\pageref{page:Different-ways-posi} by centering with quantities depending on $M$. Ignoring the lower order terms, we have uniformly over all $M\subseteq\{1,2,\ldots,d\}$,
\[
\max\{\mathcal{D}_M^{\Sigma},\; \|\Sigma_M^{-1/2}(\hat{\Gamma}_M - \hat{\Sigma}_M\beta_M)\|\} = O_p(1)\sqrt{\frac{|M|\log(ed/|M|)}{n}},
\]
which also provides $\eta_M$ for an application of Corollary~\ref{cor:BEOLS-Varaible-Selection}. It is noteworthy that we only require finite number of moments on the response.
\section{Simulation Results}\label{sec:Simulations}
We consider three different settings and compare different ways of post-selection inference as described in Page~\pageref{page:Different-ways-posi}. We only consider the case of fixed design under the well-specified linear model (that unfortunately goes against the philosophy of the paper) which we do since the lower bound and worst case results in post-selection inference are only available for fixed design case~\citep{Berk13}. The fixed design for each of the cases are as follows:
\begin{itemize}
	\item[(a)] \textbf{Orthogonal design.} We take $x_1,\ldots,x_n$ such that $\hat{\Sigma} = n^{-1}\sum_{i=1}^n x_ix_i^{\top} = I_d.$
	We find the $x_i$ by first taking a matrix $\mathcal{X}\in\mathbb{R}^{n\times d}$ satisfying $\mathcal{X}^{\top}\mathcal{X} = I_d$ and then multiply this matrix with $\hat{\Sigma}^{1/2}$ (which in this setting is $I_d$).
	\item[(b)] \textbf{Exchangeable design.} We take $x_1,\ldots,x_n$ such that $\hat{\Sigma} = I_d + \alpha\mathbf{1}_d\mathbf{1}_d^{\top}$ with $\alpha = -{1}/{(d+2)}.$ Here $\mathbf{1}_d$ is the all $1$'s vector of dimension $d$.
	\item[(c)] \textbf{Worst-case design.} We take $x_1,\ldots,x_n$ such that
	\[
	\hat{\Sigma} := \begin{bmatrix}I_{d-1} & c\mathbf{1}_{d-1}\\\mathbf{0}_{d-1}^{\top} & \sqrt{1 - (d-1)c^2}\end{bmatrix},\mbox{ with }c^2 = \frac{1}{2(d-1)},
	\]
\end{itemize}
For the first two settings, it is known that the maximum statistic~\eqref{eq:Maximum-Normal-Approximation} is of order $\sqrt{\log(d)}$ and for the last setting it is known that it is of order $\sqrt{d}$ (where we hope the other ways of PoSI would help improve the confidence intervals). See~\citet[Section 6]{Berk13} for details. For each setting, the model is
$Y_i = x_i^{\top}\beta_0 + \varepsilon_i,$
with $\varepsilon_i\overset{iid}{\sim} N(0, \sigma^2), \sigma = 1$ and $\beta_0$ is randomly generated as a vector with each coordinate being a $\texttt{Unif}(-1,1)$ independently. We consider $d = 20, \mathcal{M} = \mathcal{M}_{\le 10}, \alpha = 0.05$ (confidence level is $0.95$). Even though the variance of $\hat{\beta}_M - \beta_M$ is $\sigma^2\hat{\Sigma}_M^{-1}$, we estimate it by using~\eqref{eq:Conservative-Variance} ignoring the Gaussian response knowledge. 

We report the simulations in the following way: For all designs, we split all models in $\mathcal{M}$ into models of different sizes. We compute the (average over 500 simulations) coverage for all models of a given size and minimum, median as well as maximum (average) confidence interval length for that model size, that is,
\begin{equation}\label{eq:Reporting-quantities}
\mathbb{P}\left(\bigcap_{M\in\mathcal{M},\,|M| = s}\{\beta_M\in\hat{\mathcal{R}}_M\}\right),\quad \left\{\min_{|M| = s}, \med_{|M| = s}, \max_{|M| = s}\right\}\mathbf{m}(\hat{\mathcal{R}}_{M}),
\end{equation}
are reported with $\hat{\mathcal{R}}_M$ replaced by $\hat{\mathcal{R}}_M^{(j)}, 1\le j\le 3$ given in Page~\pageref{page:Different-ways-posi}, where $\mathbf{m}(\hat{\mathcal{R}}_{M})$ represents the threshold of the confidence region for model $M$ (e.g., for $\hat{\mathcal{R}}_M^{(1)}$ it is $\widehat{\med}(T_M) + K_{\mathcal{M}}^{(1)}(\alpha)$). Note that this threshold is a proxy for the volume of the confidence region. Additionally we consider $\hat{\mathcal{R}}_M^{(0)}$ given by
\[
\left\{\theta\in\mathbb{R}^{|M|}:\,\max_{j}\left|\frac{\hat{\beta}_{M,j} - \theta_j}{\hat{\sigma}_{M,j}}\right| \le K_{\mathcal{M}}^{(0)}(\alpha)\right\}\;\mbox{with}\; K_{\mathcal{M}}^{(0)}(\alpha) := (1-\alpha)\mbox{-quantile}\left(\max_{M\in\mathcal{M}}T_M\right). 
\]
Finally we also report $\mathbb{P}(\cap_{M\in\mathcal{M}}\{\beta_M\in\hat{\mathcal{R}}_M\})$. Note that by construction this probability has to be about $0.95$ and by noting the first quantity in~\eqref{eq:Reporting-quantities}, we see if the constructed confidence regions are too conservative for models of smaller sizes. Table~\ref{tab:TotalCoverage} shows the average coverage from all methods in all settings confirming that these are valid post-selection confidence regions. 
\begin{table}[!h]
\centering
\begin{tabular}{lrrr}
  \hline
method & Setting (a) & Setting (b) & Setting (c) \\ 
  \hline
method 0 & 0.986 & 0.976 & 0.972 \\ 
  method 1 & 0.964 & 0.960 & 0.964 \\ 
  method 2 & 0.964 & 0.956 & 0.958 \\ 
  method 3 & 0.964 & 0.956 & 0.958 \\ 
   \hline
\end{tabular}
\caption{\label{tab:TotalCoverage}The numbers in table represent the average simultaneous coverage of all confidence regions for all settings estimate of $\mathbb{P}(\cap_{M\in\mathcal{M}}\{\beta_M\in\hat{\mathcal{R}}_M\})$ based on 500 replications.}
\end{table}

Figures~\ref{fig:Settinga},~\ref{fig:Settingb}, and~\ref{fig:Settingc} show the results (for settins (a), (b) and (c), respectively) from 500 simulations within each 200 bootstrap samples were used. In all the settings, the coverage from the proposed methods ($\hat{\mathcal{R}}_M^{(j)}, j = 1,2,3$) is closer to $0.95$ and for many models the proposed intervals are shorter than the ones from $\hat{\mathcal{R}}_M^{(0)}$. 
\begin{figure}[!h]
\includegraphics[width=\textwidth]{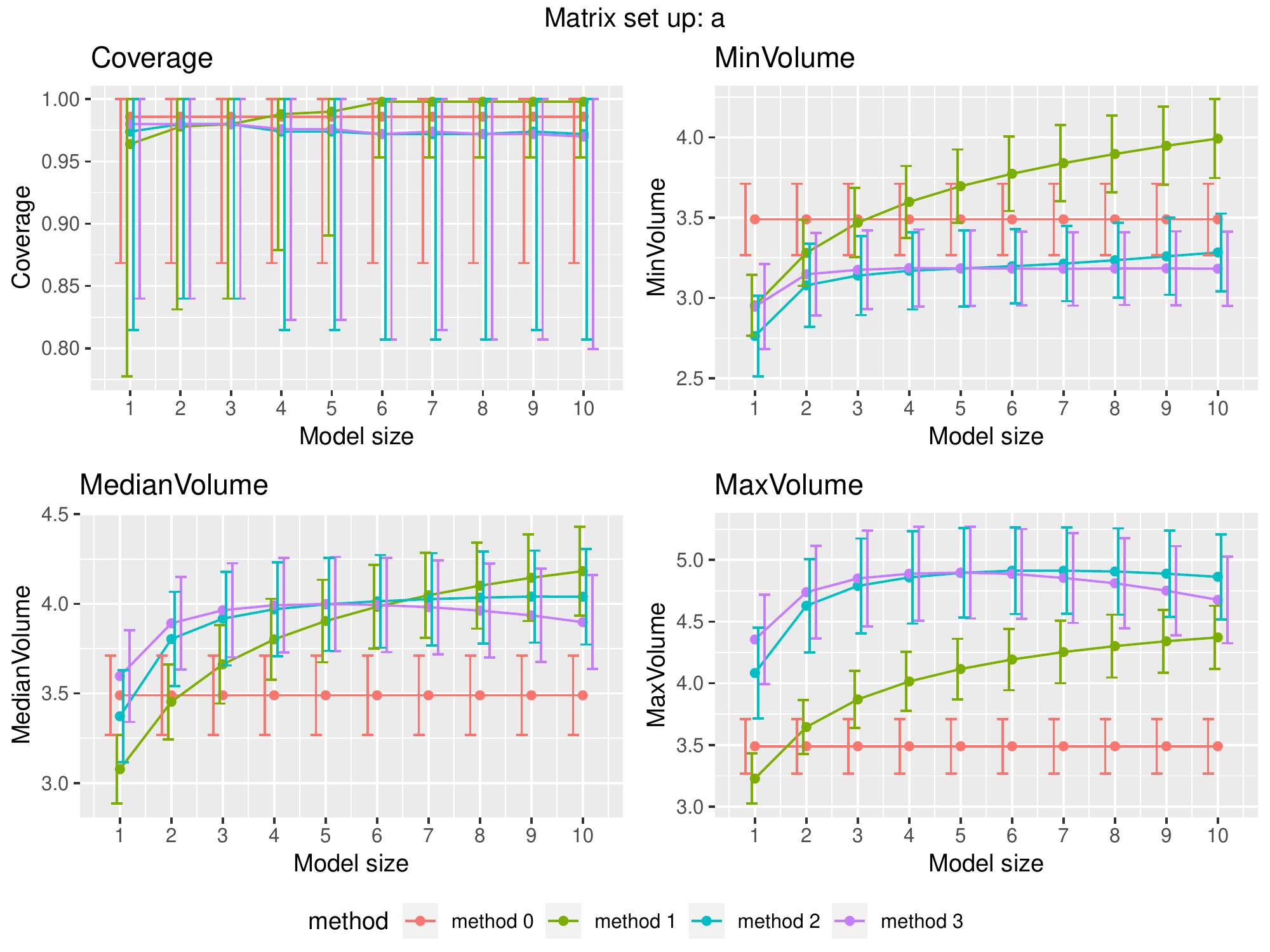}
\caption{The results for setting (a) with orthogonal design. The lines represents average (of quantities in~\eqref{eq:Reporting-quantities}) over 500 replications and error bars are $\pm 1$ SD over replications. Method $j$ in legend refers to confidence regions $\hat{\mathcal{R}}_M^{(j)}$ for $j = 0, 1, 2, 3$. Volume in the plots refers to the threshold $\textbf{m}(\hat{\mathcal{R}}_M)$.}
\label{fig:Settinga}
\end{figure}
\begin{figure}[!h]
\includegraphics[width=\textwidth]{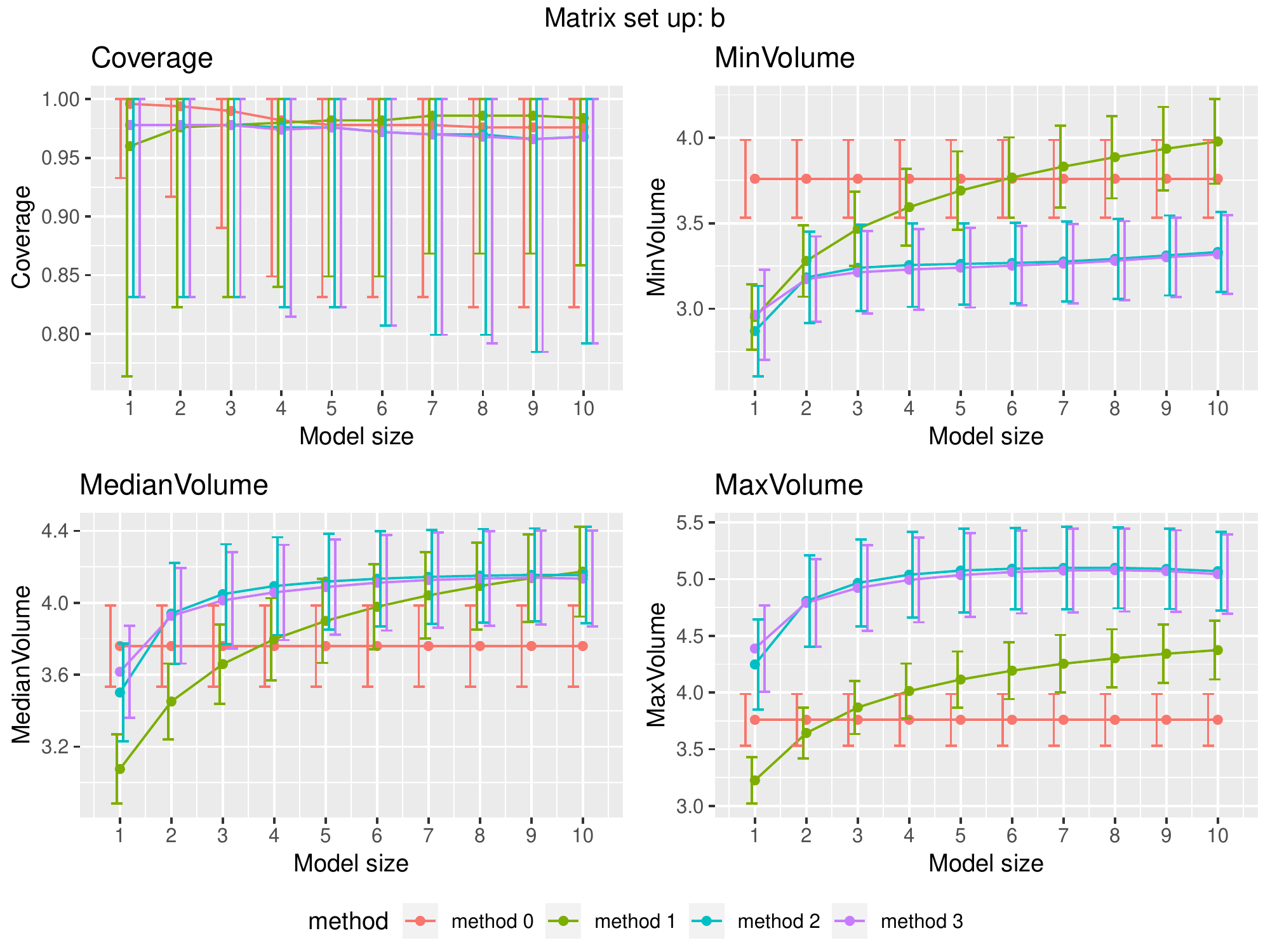}
\caption{The results for setting (b) with exhangeable design.}
\label{fig:Settingb}
\end{figure}
\begin{figure}[!h]
\includegraphics[width=\textwidth]{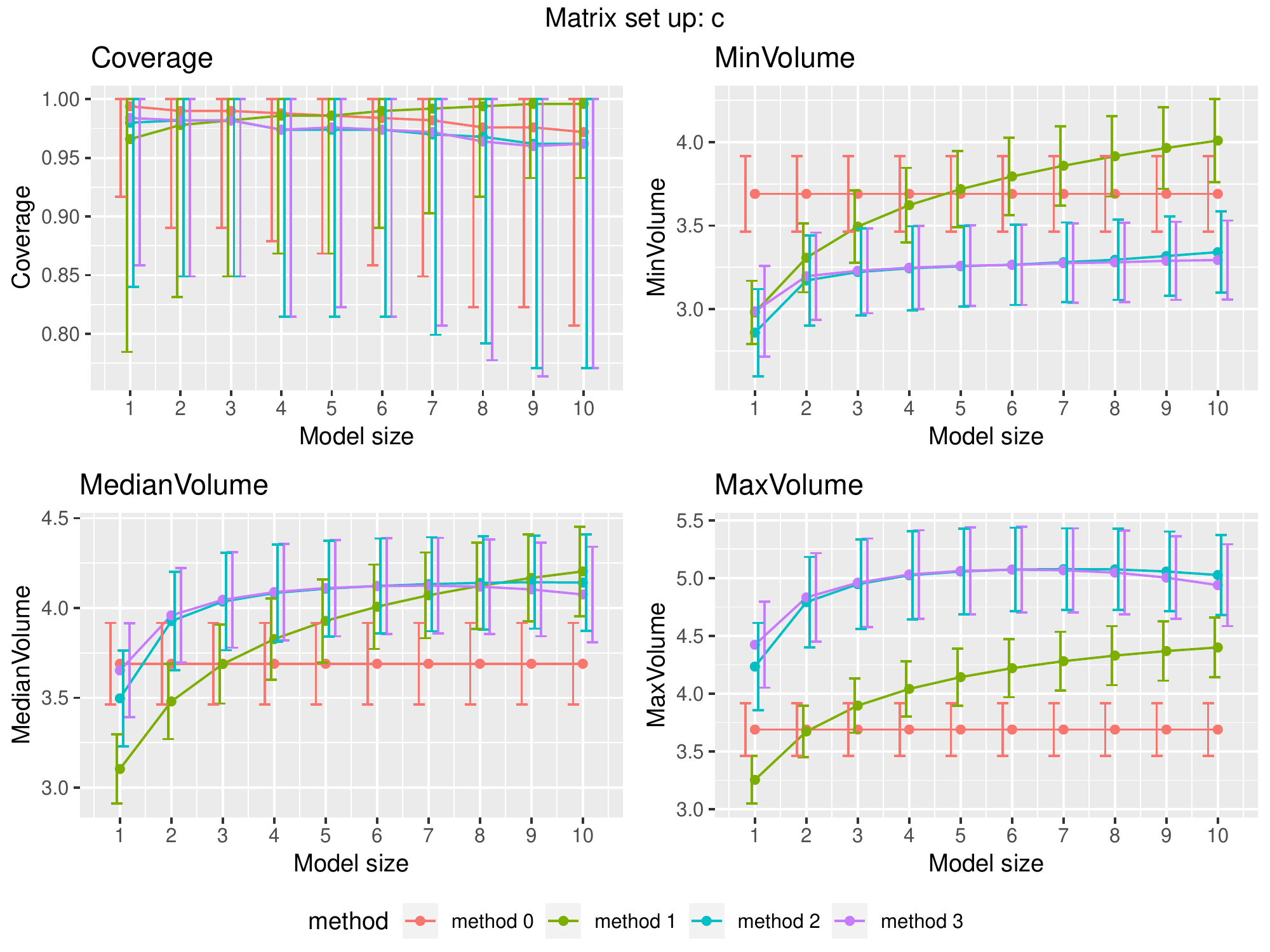}
\caption{The results for setting (c) with worst case design.}
\label{fig:Settingc}
\end{figure}
\section{Summary and Final Word}\label{sec:Summary}
We have provided a completely deterministic study of ordinary least squares linear regression setting which implies asymptotic normality, inference, inference under variable selection and much more without requiring any of the classical model assumptions. This study brings out two important quantities that needs to be controlled for a complete study of the OLS estimator. We control these quantities in case of independent observations allowing for the total number of covariates to diverge with the sample size (almost exponentially).

We have shown through our results here that the study of an estimator can be split into two parts. One that leads to (deterministic) inequalities that hold for any set of observations and one that requires assumptions on data generating process to control the remainder terms in the deterministic inequalities or Berry--Esseen type results or (more importantly) for inference. We have extensively studied the first part in this paper and the second part (inferential part) needs to be understood more carefully when the observations are dependent; the references mentioned about block bootstrap/resampling techniques would be a starting point but rates in finite samples with increasing dimensions needs to be understood.

In the later part of the paper, we have focused on OLS under variable selection. From the derivation it should be clear that variable selection is just a choice we made and one can easily study OLS under transformations of response and/or covariates using the deterministic inequality.

We have chosen to study the OLS linear regression estimator because of its simplicity; even in this case some calculations get messy. An almost parallel set of results can be derived for other regression estimators including GLMs, Cox proportional hazards model and so on; see~\cite{2018arXiv180905172K} for details.

Finally we close with some comments on computation. The methods of inference after variable selection mentioned in Section~\ref{subsec:Inference-OLS-Variable-Selection} involve computing maximum over all models $|M| = s$ and there are $\binom{d}{s}$ many such. This can be prohibitive if $d$ or $s$ is large. Allowing for slightly enlarged confidence regions (conservative inference), one can try to approximate these maximums from above without exact computation. We now briefly discuss one way of doing this and details are left for a future work. Suppose we want to find the maximum norm of $w = (w_1,\ldots,w_m)\in\mathbb{R}_+^m$. Further suppose we know an upper bound, $B$, on $\|w\|_{\infty}$. Note the trivial inequality
\[\textstyle
\left(m^{-1}\sum_{j=1}^m w_i^q\right)^{1/q} ~\le~ \|w\|_{\infty} ~\le~ m^{1/q}\left(m^{-1}\sum_{j=1}^m w_i^q\right)^{1/q},
\] 
for any $q\ge1$. If $q = \log(m)/\varepsilon$, then $\|w\|_{\infty}$ is $(m^{-1}\sum_{j=1}^m w_i^q)^{1/q}$ up to a factor of $e^{\varepsilon}$. Observe now that $(m^{-1}\sum_{j=1}^m w_i^q) = \mathbb{E}_{J}[w_J^q]$ (for $J\sim\texttt{Unif}\{1,\ldots,m\}$) is an expectation which can be estimated by $k^{-1}\sum_{\ell = 1}^k w_{j_{\ell}}^q$ for $j_1,\ldots,j_k\overset{iid}{\sim} \texttt{Unif}\{1,\ldots,m\}$. This is only an estimator of the expectation but using the apriori upper bound $B$, one can use any of the existing concentration inequalities to get a finite sample confidence interval for $(m^{-1}\sum_{j=1}^m w_i^q)^{1/q}$ which leads to an upper estimate of $\|w\|_{\infty}$. The details such as ``which concentration inequality is good?, how good the upper bound is?'' will be given elsewhere.      
\bibliographystyle{apalike}
\bibliography{../AssumpLean}
\appendix
\section{Proof of Corollary~\ref{cor:BEOLS}}\label{AppSec:ProofBEOLS}
\begin{proof}
From the definition of $\Delta_n$ and Theorem 2.1 of~\cite{rudelson2013hanson}, we get for all $r > 0$,
\begin{align*}
\mathbb{P}\left(\|\Sigma^{-1}(\hat{\Gamma} - \hat{\Sigma}\beta)\|_{\Sigma} > r + \|K^{1/2}\|_{HS}\right) &\le \mathbb{P}\left(\|K^{1/2}N(0, I_d)\|_{2} > r + \|K^{1/2}\|_{HS}\right) + \Delta_n
\\
&
\le 2\exp\left(-\frac{c_1^2r^2}{\|K^{1/2}\|_{op}^2}\right) + \Delta_n,
\end{align*}
for some constant $c_1 > 0$ (independent of $p$ and $n$). Thus, we get for all $n\ge 1$,
\begin{equation}\label{eq:BoundEstimErr}
\mathbb{P}\left(\|\Sigma^{-1}(\hat{\Gamma} - \hat{\Sigma}\beta)\|_{\Sigma} > c_1^{-1}\|K^{1/2}\|_{op}\sqrt{\log n} + \|K^{1/2}\|_{HS}\right) \le 2n^{-1} + \Delta_n.
\end{equation}
For any set $A\subseteq\mathbb{R}^p$ and $\epsilon > 0$, let $A^{\epsilon}$ denote the $\epsilon$-inflation of the set $A$ with respect to the norm $\|\cdot\|_{\Sigma}$, that is, $A^{\epsilon} := \left\{y\in\mathbb{R}^p:\,\|y - x\|_{\Sigma} \le \epsilon\mbox{ for some }x\in A\right\}.$
Using Theorem~\ref{thm:Deterministic-Ineq}, we get with $\mathcal{D}^{\Sigma}$ as in~\eqref{eq:D-Sigma-Definition}, for any set $A\subseteq\mathbb{R}^p$,
\begin{align*}
\mathbb{P}\left(\Sigma^{1/2}(\hat{\beta} - \beta) \in A\right) &\le \mathbb{P}\left(\Sigma^{-1/2}(\hat{\Gamma} - \hat{\Sigma}\beta)\in A^{r_n\eta}\right)\\ &\qquad+ \mathbb{P}\left(\|\Sigma^{-1}(\hat{\Gamma} - \hat{\Sigma}\beta)\|_{\Sigma} > r_n\right) + \mathbb{P}\left(\mathcal{D}^{\Sigma} > \eta\right),\\
\mathbb{P}\left(\Sigma^{-1/2}(\hat{\Gamma} - \hat{\Sigma}\beta)\in A\right) &\le \mathbb{P}\left(\Sigma^{1/2}(\hat{\beta} - \beta) \in A^{r_n\eta}\right)\\
&\qquad+ \mathbb{P}\left(\|\Sigma^{-1}(\hat{\Gamma} - \hat{\Sigma}\beta)\|_{\Sigma} > r_n\right) + \mathbb{P}\left(\mathcal{D}^{\Sigma} > \eta\right).
\end{align*}
Therefore, we get
\begin{align*}
&\left|\mathbb{P}\left(\Sigma^{1/2}(\hat{\beta} - \beta)\in A\right) - \mathbb{P}\left(\Sigma^{-1/2}(\hat{\Gamma} - \hat{\Sigma}\beta) \in A\right)\right|\\ &\quad\le \mathbb{P}\left(\Sigma^{-1/2}(\hat{\Gamma} - \hat{\Sigma}\beta) \in A^{r_n\eta}\setminus A\right) + \mathbb{P}\left(\mathcal{D}^{\Sigma} > \eta\right) + \mathbb{P}\left(\|\Sigma^{-1}(\hat{\Gamma} - \hat{\Sigma}\beta)\|_{\Sigma} > r_n\right).
\end{align*}
Additionally from the definition of $\Delta_n$, we get for any convex set $A\subseteq\mathbb{R}^p$,
\begin{align*}
&\left|\mathbb{P}\left(\Sigma^{1/2}(\hat{\beta} - \beta)\in A\right) - \mathbb{P}\left(\Sigma^{-1/2}(\hat{\Gamma} - \hat{\Sigma}\beta) \in A\right)\right|\\ &\quad\le \mathbb{P}\left(K^{1/2}N(0, I_d) \in A^{r_n\eta}\setminus A\right) + 2\Delta_n + \mathbb{P}\left(\mathcal{D}^{\Sigma} > \eta\right) + \mathbb{P}\left(\|\Sigma^{-1}(\hat{\Gamma} - \hat{\Sigma}\beta)\|_{\Sigma} > r_n\right).
\end{align*}
Recall that $N(0, I_d)$ represents a standard normal random vector. Now we get, from Lemma 2.6 of~\cite{bentkus2003dependence} and the discussion following, that there exists a constant $c_2 > 0$ such that
$\sup_{A\in\mathcal{C}_d}\mathbb{P}\left(K^{1/2}N(0, I_d) \in A^{r_n\eta}\setminus A\right) \le c_2\|K^{-1}\|_{*}^{1/4}r_n\eta,$
where $\|M\|_{*}$ for a matrix $M\in\mathbb{R}^{p\times p}$ denotes the nuclear norm of $M$. Hence
\begin{align*}
&\sup_{A\in\mathcal{C}_d}\left|\mathbb{P}\left(\Sigma^{1/2}(\hat{\beta} - \beta)\in A\right) - \mathbb{P}\left(\Sigma^{-1/2}(\hat{\Gamma} - \hat{\Sigma}\beta) \in A\right)\right|\\ &\quad\le c_2\|K^{-1}\|_*^{1/4}r_n\eta + 2n^{-1} + 3\Delta_n + \mathbb{P}\left(\mathcal{D}^{\Sigma} > \eta\right).
\end{align*}
Here we have used inequality~\eqref{eq:BoundEstimErr}. Finally, from the definition of $\Delta_n$, we get
\begin{align*}
&\sup_{A\in\mathcal{C}_d}\left|\mathbb{P}\left(\Sigma^{1/2}(\hat{\beta} - \beta)\in A\right) - \mathbb{P}\left(K^{1/2}N(0, I_d) \in A\right)\right|\\ &\quad\le c_2\|K^{-1}\|_*^{1/4}r_n\eta + 2n^{-1} + 4\Delta_n + \mathbb{P}\left(\mathcal{D}^{\Sigma} > \eta\right).
\end{align*}
Since $\mathcal{C}_d$ is invariant under linear transformations, the result follows.
\end{proof}
\section{Proof of Corollary~\ref{cor:BEOLS-Varaible-Selection}}\label{AppSec:ProofBEOLS-Variable-Selection}
We first prove a version of Corollary~\ref{cor:Uniform-in-Submodel} for the purpose of normal approximation with $\|\cdot\|_{\Sigma_M}$ replaced by $\|\cdot\|_{\Sigma_MV_M^{-1}\Sigma_M}$. We start with equality before~\eqref{eq:Almost-Final-Bound} in the proof of Theorem~\ref{thm:Deterministic-Ineq} for model $M$:
\[
\Sigma_M^{1/2}\left[\hat{\beta}_M - \beta_M - \Sigma_M^{-1}(\hat{\Gamma}_M - \hat{\Sigma}_M\beta_M)\right] ~=~ (I_{|M|} - \Sigma_M^{-1/2}\hat{\Sigma}_M\Sigma_M^{-1/2})\Sigma_M^{1/2}(\hat{\beta}_M - \beta_M).
\]
Multiplying both sides by $V_M^{-1/2}$ and applying Euclidean norm, we get
\begin{align*}
&\|\hat{\beta}_M - \beta_M - \Sigma_M^{-1}(\hat{\Gamma}_M - \hat{\Sigma}_M\beta_M)\|_{\Sigma_MV_M^{-1}\Sigma_M}\\ 
&\qquad\le \|V_M^{-1/2}(I_{|M|} - \Sigma_M^{-1/2}\hat{\Sigma}_M\Sigma_M^{-1/2})V_M^{1/2}V_M^{-1/2}\Sigma_M^{1/2}(\hat{\beta}_M - \beta_M)\|\\
&\qquad\le \|V_M^{-1/2}(I_{|M|} - \Sigma_M^{-1/2}\hat{\Sigma}_M\Sigma_M^{-1/2})V_M^{1/2}\|_{op}\|\hat{\beta}_M - \beta_M\|_{\Sigma_MV_M^{-1}\Sigma_M}\\
&\qquad= \mathcal{D}_M^{\Sigma}\|\hat{\beta}_M - \beta_M\|_{\Sigma_MV_M^{-1}\Sigma_M}.
\end{align*}
The last equality above follows from the fact that $\|AB\|_{op} = \|BA\|_{op}$. This implies
\begin{equation}\label{eq:Reformulation-corollary-uniform}
\|\hat{\beta}_M - \beta_M - \Sigma_M^{-1}(\hat{\Gamma}_M - \hat{\Sigma}_M\beta_M)\|_{\Sigma_MV_M^{-1}\Sigma_M} \le \frac{\mathcal{D}^{\Sigma}_M}{(1 - \mathcal{D}_M^{\Sigma})_+}\|\Sigma_M^{-1}(\hat{\Gamma}_M - \hat{\Sigma}_M\beta_M)\|_{\Sigma_MV_M^{-1}\Sigma_M}.
\end{equation}
Observe now that for any $x\in\mathbb{R}^{|M|}$ and any invertible matrix $A$,
\begin{align}\label{eq:Scaled-Euclidean-Maximum-Comparison}
\begin{split}
\|x\|_A = \|A^{1/2}x\| = \max_{\theta\in\mathbb{R}^{|M|}}\frac{\theta^{\top}x}{\sqrt{\theta^{\top}A^{-1}\theta}} \ge \max_{\substack{\theta=\pm e_j,\\1\le j\le |M|}}\frac{|\theta^{\top}x|}{\sqrt{\theta^{\top}A^{-1}\theta}} = \max_{1\le j\le |M|}\,\frac{|x_j|}{\sqrt{(A^{-1})_j}}. 
\end{split}
\end{align}
Therefore, combining~\eqref{eq:Reformulation-corollary-uniform} and~\eqref{eq:Scaled-Euclidean-Maximum-Comparison}, we get for all $M\in\mathcal{M}$,
\[
\max_{1\le j\le |M|}\,\frac{|(\hat{\beta}_M - \beta_M - \Sigma_M^{-1}(\hat{\Gamma}_M - \hat{\Sigma}_M\beta_M))_j|}{\sqrt{(\Sigma_M^{-1}V_M\Sigma_M^{-1})_j}} \le \frac{\mathcal{D}_M^{\Sigma}\|\Sigma_M^{-1}(\hat{\Gamma}_M - \hat{\Sigma}_M\beta_M)\|_{\Sigma_MV_M^{-1}\Sigma_M}}{(1 - \mathcal{D}_M^{\Sigma})_+}.
\]
From the definition of the $1/2$-net, it follows that
\[
\|\Sigma_M^{-1}(\hat{\Gamma}_M - \hat{\Sigma}_M\beta_M)\|_{\Sigma_MV_M^{-1}\Sigma_M} \le 2\max_{\theta\in\mathcal{N}_{|M|}^{1/2}} \theta^{\top}V_M^{-1/2}(\hat{\Gamma}_M - \hat{\Sigma}_M\beta_M).
\]
See, e.g.,~\citet[Theorem 1.19]{rigollet2015high}. Therefore, for all $M\in\mathcal{M}$,
\begin{equation}\label{eq:First-maximum-Bound-L2-on-right}
\max_{1\le j\le |M|}\,\frac{|(\hat{\beta}_M - \beta_M - \Sigma_M^{-1}(\hat{\Gamma}_M - \hat{\Sigma}_M\beta_M))_j|}{\sqrt{(\Sigma_M^{-1}V_M\Sigma_M^{-1})_j}} \le \frac{2\mathcal{D}_M^{\Sigma}\max_{\theta\in\mathcal{N}_{|M|}^{1/2}}\,\theta^{\top}V_M^{-1/2}(\hat{\Gamma}_M - \hat{\Sigma}_M\beta_M)}{(1 - \mathcal{D}_M^{\Sigma})_+}.
\end{equation}
Using the definition of $\Xi_{n,\mathcal{M}}$, we can control $\max_{\theta\in\mathcal{N}_{|M|}^{1/2}}\theta^{\top}V_M^{-1/2}(\hat{\Gamma}_M - \hat{\Sigma}_M\beta_M)$. Observe first that
\begin{align}\label{eq:First-union-bound}
\begin{split}
&\mathbb{P}\left(\max_{M\in\mathcal{M}}\max_{\theta\in\mathcal{N}_{|M|}^{1/2}}\frac{\theta^{\top}\bar{G}_{M}}{\sqrt{2\log(|\mathcal{M}|5^{|M|}\pi_{|M|}) + 2\log(|M|^2/\Xi_{n,\mathcal{M}})}} \ge 1\right)\\
&\qquad\le \sum_{s = 1}^d \mathbb{P}\left(\max_{M\in\mathcal{M}, |M| = s}\max_{\theta\in\mathcal{N}_{s}^{1/2}}{\theta^{\top}\bar{G}_M} \ge \sqrt{2\log(|\mathcal{M}|5^{s}\pi_{s}) + 2\log(s^2/\Xi_{n,\mathcal{M}})}\right).
\end{split}
\end{align}
Since $\bar{G}_M$ is a standard normal random vector for each $M\in\mathcal{M}$, $\theta^{\top}\bar{G}_{M}$ is a standard Gaussian random variable and it follows from~\citet[Theorem 1.14]{rigollet2015high} that for all $t\ge0$,
\[
\mathbb{P}\left(\max_{M\in\mathcal{M},|M| = s}\,\max_{\theta\in\mathcal{N}_{|M|}^{1/2}}\theta^{\top}\bar{G}_M \ge \sqrt{2\log(|\mathcal{M}|5^s\pi_s) + 2t}\right) \le \exp(-t),
\]
Taking $t = \log(s^2/\Delta_{n,M})$ yields
\[
\mathbb{P}\left(\max_{M\in\mathcal{M}, |M| = s}\max_{\theta\in\mathcal{N}_{s}^{1/2}}{\theta^{\top}\bar{G}_M} \ge \sqrt{2\log(|\mathcal{M}|5^{s}\pi_{s}) + 2\log(s^2/\Xi_{n,\mathcal{M}})}\right) \le \frac{\Xi_{n,\mathcal{M}}}{s^2}.
\]
Combining this with~\eqref{eq:First-union-bound} and using $\sum_{s = 1}^d s^{-2} \le \pi^2/6 < 1.65$, we get
\[
\mathbb{P}\left(\max_{M\in\mathcal{M}}\max_{\theta\in\mathcal{N}_{|M|}^{1/2}}\frac{\theta^{\top}\bar{G}_{M}}{\sqrt{2\log(|\mathcal{M}|5^{|M|}\pi_{|M|}) + 2\log(|M|^2/\Xi_{n,\mathcal{M}})}} \ge 1\right) \le 1.65\Xi_{n,\mathcal{M}}.
\]
From the definition of $\Xi_{n,M}$, it follows that
\[
\mathbb{P}\left(\max_{M\in\mathcal{M}}\max_{\theta\in\mathcal{N}_{|M|}^{1/2}}\frac{\theta^{\top}V_M^{-1/2}(\hat{\Gamma}_M - \hat{\Sigma}_M\beta_M)}{\sqrt{2\log(|\mathcal{M}|5^{|M|}\pi_{|M|}) + 2\log(|M^2|/\Xi_{n,\mathcal{M}})}} > 1\right) \le 2.65\Xi_{n,\mathcal M}.
\]
Hence for any $(\eta_M)_{M\in\mathcal{M}} (\le 1/2)$, on an event with probability at least $1 - 2.65\Xi_{n,M} - \mathbb{P}(\cup_{M\in\mathcal{M}}\{\mathcal{D}_M^{\Sigma} \ge \eta_M\})$, we get
\begin{equation}\label{eq:Almost-final-variable-selection-normal-approx}
\max_{1\le j\le |M|}\,\frac{|(\hat{\beta}_M - \beta_M - \Sigma_M^{-1}(\hat{\Gamma}_M - \hat{\Sigma}_M\beta_M))_j|}{\sqrt{(\Sigma_M^{-1}V_M\Sigma_M^{-1})_j}} \le {4\eta_M\sqrt{2\log(|\mathcal{M}|5^{|M|}|M|^2\pi_{|M|}/\Xi_{n,M})}}.
\end{equation}
Define a vector $\varepsilon\in\mathbb{R}^{\sum_{M\in\mathcal{M}}|M|}$ indexed by $M\in\mathcal{M}, 1\le j\le |M|$ such that 
\[
\varepsilon_{M,j} := {4\eta_M\sqrt{2\log(|\mathcal{M}|5^{|M|}|M|^2\pi_{|M|}/\Xi_{n,M})}}.
\]
Fix any set $A\in\mathcal{A}^{sre}$. Then from~\eqref{eq:Almost-final-variable-selection-normal-approx}, we get
\begin{align*}
\mathbb{P}\left(\left(\frac{(\hat{\beta}_M - \beta_M)_j}{\sqrt{(\Sigma_M^{-1}V_M\Sigma_M^{-1})_j}}\right)_{\substack{M\in\mathcal{M},\\1\le j\le |M|}} \in A\right) &\le \mathbb{P}\left(\left(\frac{(\Sigma_M^{-1}(\hat{\Gamma}_M - \hat{\Sigma}_M\beta_M))_j}{\sqrt{(\Sigma_M^{-1}V_M\Sigma_M^{-1})_j}}\right)_{\substack{M\in\mathcal{M},\\1\le j\le |M|}} \in A + \varepsilon\right)\\
&\qquad+ 2.65\Xi_{n,M} + \mathbb{P}\left(\bigcup_{M\in\mathcal{M}}\,\{\mathcal{D}_M^{\Sigma} \ge \eta_M\}\right),
\end{align*}
and
\begin{align*}
\mathbb{P}\left(\left(\frac{(\hat{\beta}_M - \beta_M)_j}{\sqrt{(\Sigma_M^{-1}V_M\Sigma_M^{-1})_j}}\right)_{\substack{M\in\mathcal{M},\\1\le j\le |M|}} \in A\right) &\ge \mathbb{P}\left(\left(\frac{(\Sigma_M^{-1}(\hat{\Gamma}_M - \hat{\Sigma}_M\beta_M))_j}{\sqrt{(\Sigma_M^{-1}V_M\Sigma_M^{-1})_j}}\right)_{\substack{M\in\mathcal{M},\\1\le j\le |M|}} \in A - \varepsilon\right)\\
&\qquad- 2.65\Xi_{n,M} - \mathbb{P}\left(\bigcup_{M\in\mathcal{M}}\,\{\mathcal{D}_M^{\Sigma} \ge \eta_M\}\right),
\end{align*}
Hence the result follows from the definition of $\Delta_{n,\mathcal{M}}$.
\section{Proof of Proposition~\ref{prop:Rates-D_M-Gamma_M}}\label{AppSec:Proof-of-prop-Rates-D_M-Gamma_M}
Observe that 
\begin{equation}\label{eq:Operator-to-finite-maximum}
\mathcal{D}_M^{\Sigma} = \|\Sigma_M^{-1/2}\hat{\Sigma}_M\Sigma_M^{-1/2} - I_{|M|}\|_{op} \le 2\sup_{\nu\in\mathcal{N}^{1/4}_{|M|}}\left|\frac{1}{n}\sum_{i=1}^n (\nu^{\top}\Sigma_M^{-1/2}X_{i,M})^2 - 1\right|,
\end{equation}
where $\mathcal{N}^{1/4}_{|M|}$ represents the $1/4$-net of $\{\theta\in\mathbb{R}^{|M|}:\,\|\theta\| = 1\}$; see Lemma 2.2 of~\cite{Ver12}. Note that $|\mathcal{N}^{1/4}_{|M|}| \le 9^{|M|}$. Therefore the right hand side of~\eqref{eq:Operator-to-finite-maximum} is a maximum over a finite number of mean zero averages with summands satisfying
\[
\mathbb{E}\left[\exp\left({\mathfrak{K}_{\beta}^{-\beta}|\nu^{\top}\Sigma_M^{-1/2}X_{i,M}|^{\beta}}\right)\right] \le 2,\mbox{ for all }\nu\in\mathcal{N}^{1/4}_{|M|}\mbox{ and }M\subseteq\{1,2,\ldots,d\}.
\]
Applying Theorem 3.4 of~\cite{KuchAbhi17}, we get for any $t\ge0$ that with probability $1 - 3e^{-t}$,
\[
\mathcal{D}_M^{\Sigma} \le 14\sqrt{\frac{\kappa_M^{\Sigma}(t + |M|\log(9))}{n}} + \frac{C_{\beta}\mathfrak{K}_{\beta}^2(\log(2n))^{2/\beta}(t + |M|\log(9))^{\max\{1,2/\beta\}}}{n},
\]
for some constant $C_{\beta} > 0$ depending only $\beta$.
Since there are $\binom{d}{s} \le (ed/s)^s$ models of size $s$, taking $t = s\log(ed/s) + u$ (for any $u\ge0$) and applying union bound over all models of size $s$, we get that with probability $1 - 3e^{-u}$, simultaneously for all $M\subseteq\{1,2,\ldots,d\}$ with $|M| = s$,
\[
\mathcal{D}_M^{\Sigma} \le 14\sqrt{\frac{\kappa_M^{\Sigma}(u + {s\log(9ed/s)})}{n}} + \frac{C_{\beta}\mathfrak{K}_{\beta}^{2}(\log(2n))^{2/\beta}(u + s\log(9ed/s))^{\max\{1, 2/\beta\}}}{n}.
\]
To prove the result simultaneously over all $1\le s \le d$, take $u = v + \log(\pi^2s^2/6)$ and apply union bound over $1\le s\le d$ to get with probability $1 - 3e^{-v}$ simulataneously over all $M\subseteq\{1,2,\ldots,d\}$ with $|M| = s$ for some $1\le s\le d$,
\begin{align*}
\mathcal{D}_M^{\Sigma} &\le 14\sqrt{\frac{\kappa_M^{\Sigma}(v + \log(\pi^2s^2/6) + {s\log(9ed/s)})}{n}}\\ 
&\qquad+ \frac{C_{\beta}\mathfrak{K}_{\beta}^{2}(\log(2n))^{2/\beta}(v + \log(\pi^2s^2/6) + s\log(9ed/s))^{\max\{1, 2/\beta\}}}{n}.
\end{align*}
Since $s^{-1}\log(\pi^2s^2/6) \le (2\pi/\sqrt{6})\sup_{x \ge \pi/\sqrt{6}}\,\exp(-x)x \le 1$, we get with probability $1 - 3e^{-v}$ simultaneously for any $1\le s\le d$ and for any model $M\subseteq\{1,2,\ldots,d\}$ with $|M| = s$,
\[
\mathcal{D}_M^{\Sigma} \le 14\sqrt{\frac{\kappa_M^{\Sigma}(v + {s\log(9e^2d/s)})}{n}} + \frac{C_{\beta}\mathfrak{K}_{\beta}^{2}(\log(2n))^{2/\beta}(v + s\log(9e^2d/s))^{\max\{1, 2/\beta\}}}{n}.
\]
This completes the proof of~\eqref{eq:Tail-Bound-D_M-Simultaneous}.

We now bound $\|\Sigma_M^{-1/2}(\hat{\Gamma}_M - \hat{\Sigma}_M\beta_M)\|$ simultaneously over all $M$. Observe from the definition of $\beta_M$ that
\[
0 \le \sum_{i=1}^n \mathbb{E}\left[(Y_i - X_{i,M}^{\top}\beta_M)^2\right] = \sum_{i=1}^n \mathbb{E}[Y_i^2] - \sum_{i=1}^n \mathbb{E}[(X_{i,M}^{\top}\beta_M)^2],
\]
and hence $\|\tilde{\beta}_M\| = \|\Sigma_M^{1/2}\beta_M\| \le (\sum_{i=1}^n \mathbb{E}[Y_i^2]/n)^{1/2}$. Now note that since $\mathbb{E}[\hat{\Gamma}_M - \hat{\Sigma}_M\beta_M] = 0$ (from the definition of $\beta_M$), we have
\begin{align*}
\|\Sigma_M^{-1/2}(\hat{\Gamma}_M - \hat{\Sigma}_M\beta_M)\| &= \|\Sigma_M^{-1/2}(\hat{\Gamma}_M - \mathbb{E}\hat{\Gamma}_M) - \Sigma_M^{-1/2}(\hat{\Sigma}_M - \Sigma_M)\beta_M\|\\
&\le \|\Sigma_M^{-1/2}(\hat{\Gamma}_M - \mathbb{E}\hat{\Gamma}_M)\| + \|\Sigma_M^{-1/2}(\hat{\Sigma}_M - \Sigma_M)\Sigma_M^{-1/2}\|_{op}\|\Sigma_M^{1/2}\beta_M\|\\
&\le \|\Sigma_M^{-1/2}(\hat{\Gamma}_M - \mathbb{E}\hat{\Gamma}_M)\| + \mathcal{D}_M^{\Sigma}{\textstyle(\sum_{i=1}^n \mathbb{E}[Y_i^2]/n)^{1/2}}.
\end{align*}
We have already controlled $\mathcal{D}_M^{\Sigma}$ uniformly over all models $M\subseteq\{1,2,\ldots,d\}$ and hence it is enough to control $\|\Sigma_M^{-1/2}(\hat{\Gamma}_M - \mathbb{E}\hat{\Gamma}_M)\|$. As before,
observe that
\begin{align*}
\|\Sigma_M^{-1/2}(\hat{\Gamma}_M - \mathbb{E}\hat{\Gamma}_M)\| \le 2\max_{\nu\in\mathcal{N}_{|M|}^{1/2}}\left|\frac{1}{n}\sum_{i=1}^n \left\{\nu^{\top}\tilde{X}_{i,M}Y_i - \mathbb{E}[\nu^{\top}\tilde{X}_{i,M}Y_i]\right\}\right| =: 2\mathcal{E}_M,
\end{align*}
where $\tilde{X}_{i,M} := \Sigma_M^{-1/2}X_{i,M}$. To control $\mathcal{E}_M$ we split $Y_i$ in to two parts depending on whether $\{|Y_i| \le B\}$ or $\{|Y_i| > B\}$ (for a $B$ to be chosen later). Define $Y_{i,1} = Y_i\mathbbm{1}\{|Y_i| \le B\}$, $Y_{i,2} = Y_i - Y_{i,1}$ and for $\ell = 1,2$,
\[
\mathcal{E}_{M,\ell} := \max_{\nu\in\mathcal{N}_{|M|}^{1/2}}\left|\frac{1}{n}\sum_{i=1}^n \left\{\nu^{\top}\tilde{X}_{i,M}Y_{i,\ell} - \mathbb{E}[\nu^{\top}\tilde{X}_{i,M}Y_{i,\ell}]\right\}\right|.
\]
Since $|Y_{i,1}| \le B$, we have for any $\nu\in\mathcal{N}_{|M|}^{1/2}$ and $M\subseteq\{1,2,\ldots,d\}$ that
\[
\mathbb{E}\left[\exp\left(\frac{|\nu^{\top}\tilde{X}_{i,M}Y_{i,1}|^{\beta}}{(B\mathfrak{K}_{\beta})^{\beta}}\right)\right] \le 2.
\]
Hence we get by Theorem 3.4 of~\cite{KuchAbhi17} that for any $t \ge 0$, with probability $1 - 3e^{-t}$
\[
\mathcal{E}_{M,1} \le 7\sqrt{\frac{\mathfrak{V}_M(t + |M|\log(5))}{n}} + \frac{C_{\beta}B\mathfrak{K}_\beta(\log(2n))^{1/\beta}(t + |M|\log(5))^{\max\{1,1/\beta\}}}{n}.
\]
Now following same approach as used for $\mathcal{D}_M^{\Sigma}$, we get with probability $1 - 3e^{-u}$, for any $1 \le s \le d$, for any model $M\subseteq\{1,2,\ldots,d\}$ such that $|M| = s$,
\begin{equation}\label{eq:bound-Em1}
\mathcal{E}_{M,1} \le 7\sqrt{\frac{\mathfrak{V}_M(v + s\log(5e^2d/s))}{n}} + \frac{C_{\beta}B\mathfrak{K}_{\beta}(\log(2n))^{1/\beta}(v + s\log(5e^2d/s))^{\max\{1,1/\beta\}}}{n}.
\end{equation}
To bound $\mathcal{E}_{M,2}$ simultaneously over all $M$, we take
\[
B := 8\mathbb{E}\left[\max_{1\le i\le n}|Y_i|\right] \le 8n^{1/r}\max_{1\le i\le n}\left(\mathbb{E}[|Y_i|^r]\right)^{1/r} = 8n^{1/r}K_{n,r},
\] 
which is motivated by Proposition 6.8 of~\cite{LED91}. Now consider the normalized process
\[
{\mathcal{E}}_{2,\texttt{Norm}} := \max_{1\le s\le d}\max_{|M| = s}\,\frac{n^{1/2}\mathcal{E}_{M,2}}{n^{-1/2 + 1/r}K_{n,r}\mathfrak{K}_{\beta}(s\log(5e^2d/s) + \log n)^{1/\beta}}.
\]
Observe first that $\mathcal{E}_{2,\texttt{Norm}} \le \mathcal{E}^{(1)} + \mathbb{E}[\mathcal{E}^{(1)}]$, where
\[
\mathcal{E}^{(1)} = \frac{1}{n}\sum_{i=1}^n \max_{\substack{1\le s\le d,\\|M|=s}}\max_{\nu\in\mathcal{N}_s^{1/2}}\frac{n^{1/2}|\nu^{\top}\tilde{X}_{i,M}Y_{i,2}|}{n^{-1/2 + 1/r}K_{n,r}\mathfrak{K}_{\beta}(s\log(5e^2d/s) + \log n)^{1/\beta}}.
\]
Note that $\mathcal{E}^{(1)}$ is an average of non-negative random variables and hence by the choice of $B$ above and Proposition 6.8 of~\cite{LED91}, we get
\begin{align}\label{eq:Proposition-6_8-LED}
\begin{split}
\mathbb{E}\left[\mathcal{E}^{(1)}\right] &\le 8\mathbb{E}\left[\frac{1}{n}\max_{1\le i\le n}\max_{\substack{1\le s\le d,\\|M|=s}}\max_{\nu\in\mathcal{N}_s^{1/2}}\frac{n^{1/2}|\nu^{\top}\tilde{X}_{i,M}Y_{i,2}|}{n^{-1/2 + 1/r}K_{n,r}\mathfrak{K}_{\beta}(s\log(5e^2d/s) + \log n)^{1/\beta}}\right]\\
&\le 8\mathbb{E}\left[\max_{1\le i\le n}\max_{\substack{1\le s\le d,\\|M|=s}}\max_{\nu\in\mathcal{N}_s^{1/2}}\frac{n^{-1/2}|\nu^{\top}\tilde{X}_{i,M}Y_{i}|}{n^{-1/2 + 1/r}K_{n,r}\mathfrak{K}_{\beta}(s\log(5e^2d/s) + \log n)^{1/\beta}}\right]\\
&\le 8\left\|\max_{1\le i\le n}\frac{|Y_i|}{K_{n,r}n^{1/r}}\right\|_2\left\|\max_{1\le s\le d}\max_{\substack{1\le i\le n,\\|M|=s}}\max_{\nu\in\mathcal{N}_s^{1/2}}\frac{|\nu^{\top}\tilde{X}_{i,M}|}{\mathfrak{K}_{\beta}(s\log(5e^2d/s) + \log n)^{1/\beta}}\right\|_2.
\end{split}
\end{align}
Here we use $\|W\|_2$ for a random variable $W$ to denote $(\mathbb{E}[W^2])^{1/2}$. In the second factor, the number of items in the maximum for any fixed $s$ is given by $n\binom{d}{s}5^s \le n(5ed/s)^s$ and hence from~\eqref{eq:Sub-Weibull-sub-model-Covariance-Scaled}, we get
\[
\mathbb{P}\left(\max_{\substack{1\le i\le n,\\|M|=s}}\max_{\nu\in\mathcal{N}_s^{1/2}}|\nu^{\top}\tilde{X}_{i,M}| \ge \mathfrak{K}_{\beta}(t + s\log(5ed/s) + \log(n))^{1/\beta}\right) \le 2e^{-t},
\] 
and an application of union bound over $1\le s\le d$ yields
\[
\mathbb{P}\left(\bigcup_{1\le s\le d}\left\{\max_{\substack{1\le i\le n,\\|M|=s}}\max_{\nu\in\mathcal{N}_s^{1/2}}|\nu^{\top}\tilde{X}_{i,M}| \ge \mathfrak{K}_{\beta}(t + \log(\pi^2s^2/6) + s\log(5ed/s) + \log(n))^{1/\beta}\right\}\right) \le 2e^{-t},
\]
which implies
\begin{equation}\label{eq:Tail-bound-maximum-envelope}
\mathbb{P}\left(\bigcup_{1\le s\le d}\left\{\max_{\substack{1\le i\le n,\\|M|=s}}\max_{\nu\in\mathcal{N}_s^{1/2}}|\nu^{\top}\tilde{X}_{i,M}| \ge \mathfrak{K}_{\beta}(t + s\log(5e^2d/s) + \log(n))^{1/\beta}\right\}\right) \le 2e^{-t}.
\end{equation}
Hence for a constant $C_{\beta} > 0$ (depending only on $\beta$),
\begin{equation}\label{eq:Second-factor}
\left\|\max_{1\le s\le d}\max_{\substack{1\le i\le n,\\|M|=s}}\max_{\nu\in\mathcal{N}_s^{1/2}}\frac{|\nu^{\top}\tilde{X}_{i,M}|}{\mathfrak{K}_{\beta}(s\log(5e^2d/s) + \log n)^{1/\beta}}\right\|_2 \le C_{\beta}.
\end{equation}
For the first factor in~\eqref{eq:Proposition-6_8-LED}, note that (since $r\ge 2$)
\begin{equation}\label{eq:first-factor}
\left\|\max_{1\le i\le n}\frac{|Y_i|}{K_{n,r}n^{1/r}}\right\|_2 \le \left\|\max_{1\le i\le n}\frac{|Y_i|}{K_{n,r}n^{1/r}}\right\|_r \le \left(\sum_{i=1}^n \mathbb{E}\left[\frac{|Y_i|^r}{K_{n,r}^rn}\right]\right)^{1/r} \le 1.
\end{equation}
Substituting the bounds~\eqref{eq:first-factor} and~\eqref{eq:Second-factor} in~\eqref{eq:Proposition-6_8-LED} yields
\begin{equation}\label{eq:Expectation-bound-E2}
\mathbb{E}[\mathcal{E}_{2,\texttt{Norm}}] \le 2\mathbb{E}[\mathcal{E}^{(1)}] \le C_{\beta},
\end{equation}
for a constant $C_{\beta} > 0$ (which is different from the one in~\eqref{eq:Second-factor}). 
Applying Theorem 8 of~\cite{Bouch05} now yields for every $q \ge 1$
\[
\|\mathcal{E}^{(1)}\|_q \le 2\mathbb{E}[\mathcal{E}^{(1)}] + Cq\left\|\frac{1}{n}\max_{1\le i\le n}\max_{\substack{1\le s\le d,\\|M|=s}}\max_{\nu\in\mathcal{N}_s^{1/2}}\frac{n^{1/2}|\nu^{\top}\tilde{X}_{i,M}Y_{i,2}|}{n^{-1/2 + 1/r}K_{n,r}\mathfrak{K}_{\beta}(s\log(5e^2d/s) + \log n)^{1/\beta}}\right\|_q,
\]
for some (other) absolute constant $C > 0$. This implies (using~\eqref{eq:Expectation-bound-E2}) that
\[
\|\mathcal{E}_{2,\texttt{Norm}}\|_q \le 3C_{\beta} + Cq\left\|\frac{1}{n}\max_{1\le i\le n}\max_{\substack{1\le s\le d,\\|M|=s}}\max_{\nu\in\mathcal{N}_s^{1/2}}\frac{n^{1/2}|\nu^{\top}\tilde{X}_{i,M}Y_{i,2}|}{n^{-1/2 + 1/r}K_{n,r}\mathfrak{K}_{\beta}(s\log(5e^2d/s) + \log n)^{1/\beta}}\right\|_q.
\]
As before, we have
\begin{align*}
&\left\|\frac{1}{n}\max_{1\le i\le n}\max_{\substack{1\le s\le d,\\|M|=s}}\max_{\nu\in\mathcal{N}_s^{1/2}}\frac{n^{1/2}|\nu^{\top}\tilde{X}_{i,M}Y_{i,2}|}{n^{-1/2 + 1/r}K_{n,r}\mathfrak{K}_{\beta}(s\log(5e^2d/s) + \log n)^{1/\beta}}\right\|_q\\
&\qquad\le \left\|\max_{1\le i\le n}\frac{|Y_i|}{K_{n,r}n^{1/r}}\max_{1\le i\le n}\max_{\substack{1\le s\le d,\\|M|=s}}\max_{\nu\in\mathcal{N}_s^{1/2}}\frac{|\nu^{\top}\tilde{X}_{i,M}|}{\mathfrak{K}_{\beta}(s\log(5e^2d/s) + \log n)^{1/\beta}}\right\|_q\\
&\qquad\le \left\|\max_{1\le i\le n}\frac{|Y_i|}{K_{n,r}n^{1/r}}\right\|_{r}\left\|\max_{1\le i\le n}\max_{\substack{1\le s\le d,\\|M|=s}}\max_{\nu\in\mathcal{N}_s^{1/2}}\frac{|\nu^{\top}\tilde{X}_{i,M}|}{\mathfrak{K}_{\beta}(s\log(5e^2d/s) + \log n)^{1/\beta}}\right\|_{rq/(r-q)}.
\end{align*}
where the last inequality holds for any $q < r$ by H{\"o}lder's inequality. We already have that the first factor is bounded be $1$. From~\eqref{eq:Tail-bound-maximum-envelope}, we have
\[
\left\|\max_{1\le i\le n}\max_{\substack{1\le s\le d,\\|M|=s}}\max_{\nu\in\mathcal{N}_s^{1/2}}\frac{|\nu^{\top}\tilde{X}_{i,M}|}{\mathfrak{K}_{\beta}(s\log(5e^2d/s) + \log n)^{1/\beta}}\right\|_{rq/(r-q)} \le C_{\beta}\left(\frac{rq}{r-q}\right)^{1/\beta}.
\]
Therefore taking $q = r-1$, we get
\[
\|\mathcal{E}_{2,\texttt{Norm}}\|_{r-1} \le 3C_{\beta} + CC_{\beta}(r-1)(r(r-1))^{1/\beta} =: C_{\beta, r}.
\]
Hence by Markov's inequality, we get with probability at least $1 - 1/t^{r-1}$, for any $1\le s \le d$, for any model $M\subseteq\{1,2,\ldots,d\}$ such that $|M|=s$,
\begin{equation}\label{eq:bound-Em2}
\mathcal{E}_{M,2} \le \frac{tC_{\beta, r}K_{n,r}\mathfrak{K}_{\beta}(s\log(5e^2d/s) + \log n)^{1/\beta}}{n^{1 - 1/r}}.
\end{equation}
Combining the bounds~\eqref{eq:bound-Em1} and~\eqref{eq:bound-Em2} yields: with probability at least $1 - 3e^{-t} - t^{-r+1}$, for any $1\le s\le d$, for any model $M\subseteq\{1,2,\ldots,d\}$ such that $|M| = s$,
\begin{align*}
\mathcal{E}_{M} &\le 7\sqrt{\frac{\mathfrak{V}_M(t + s\log(5e^2d/s))}{n}} + \frac{C_{\beta}K_{n,r}\mathfrak{K}_{\beta}(\log(2n))^{1/\beta}(t + s\log(5e^2d/s))^{\max\{1,1/\beta\}}}{n^{1 - 1/r}}\\
&\qquad+ \frac{tC_{\beta, r}K_{n,r}\mathfrak{K}_{\beta}(s\log(5e^2d/s) + \log n)^{1/\beta}}{n^{1 - 1/r}}.
\end{align*}
Combining all inequaliteies completes the proof of~\eqref{eq:influence-function-expansion-bound}.
\end{document}